\documentclass[a4paper,11pt,fleqn]{article}
\usepackage{enumitem}
\usepackage{amsmath}
\usepackage{amssymb}
\usepackage{pifont}
\usepackage{theorem} 
\usepackage{euscript}
\usepackage{exscale,relsize}
\usepackage{epic,eepic}
\usepackage{multicol}
\usepackage{makeidx}
\usepackage{srcltx}         
\usepackage{xcolor}
\usepackage{url}
\usepackage{charter}
\usepackage{mathrsfs} 
\usepackage{graphicx}
\usepackage{authblk}
\usepackage{tikz}
\newcommand{\email}[1]{\href{mailto:#1}{\nolinkurl{#1}}}
\usepackage[normalem]{ulem}     
\usepackage[usenames,dvipsnames]{pstricks}
\usepackage{pstricks-add}
\usepackage{pst-grad}
\usepackage{pst-plot} 
\usepackage{pst-eucl} 
\newlength{\mySubFigSize}
\definecolor{labelkey}{rgb}{0,0.08,0.45}
\definecolor{refkey}{rgb}{0,0.6,0.0}
\definecolor{Brown}{rgb}{0.45,0.0,0.05}
\definecolor{dbrown}{rgb}{0.50,0.25,0.00}
\definecolor{dgreen}{rgb}{0.00,0.49,0.00}
\definecolor{dblue}{rgb}{0,0.08,0.75}
\RequirePackage[dvips,colorlinks,hyperindex]{hyperref}
\hypersetup{linktocpage=true,citecolor=dblue,linkcolor=dgreen}
\renewcommand{\leq}{\ensuremath{\leqslant}}
\renewcommand{\geq}{\ensuremath{\geqslant}}

\newcommand{\minimize}[2]{\ensuremath{\underset{\substack{{#1}}}%
{\text{minimize}}\;\;#2 }}
\newcommand{\Frac}[2]{\displaystyle{\frac{#1}{#2}}} 
 
\newcommand{\scal}[2]{{\langle{{#1}\mid{#2}}\rangle}}

\newcommand{\menge}[2]{\big\{{#1}~\big |~{#2}\big\}} 
\newcommand{\Menge}[2]{\left\{{#1}~\left|~{#2}\right.\right\}}

\newcommand{\HH}{\ensuremath{{\mathcal H}}}

\newcommand{\GG}{\ensuremath{{\mathcal G}}}

\newcommand{\Sum}{\ensuremath{\displaystyle\sum}}
\newcommand{\emp}{\ensuremath{{\varnothing}}}

\newcommand{\Id}{\ensuremath{\operatorname{Id}}\,}
\newcommand{\ID}{\ensuremath{\boldsymbol{\operatorname{Id}}}\,}

\newcommand{\RR}{\ensuremath{\mathbb{R}}}
\newcommand{\RP}{\ensuremath{\left[0,+\infty\right[}}

\newcommand{\RPP}{\ensuremath{\left]0,+\infty\right[}}

\newcommand{\RPX}{\ensuremath{\left[0,+\infty\right]}}
\newcommand{\RX}{\ensuremath{\left]-\infty,+\infty\right]}}
\newcommand{\RXX}{\ensuremath{\left[-\infty,+\infty\right]}}

\newcommand{\NN}{\ensuremath{\mathbb N}}

\newcommand{\intdom}{\text{\text{\rm int\,dom}\,}}

\newcommand{\pinf}{\ensuremath{{+\infty}}}
\newcommand{\minf}{\ensuremath{{-\infty}}}
\newcommand{\epi}{\ensuremath{\text{\rm epi}\,}}
\newcommand{\dom}{\ensuremath{\text{\rm dom}\,}}

\newcommand{\rec}{\ensuremath{\text{\rm rec}\,}}
\newcommand{\prox}{\ensuremath{\text{\rm prox}}}
\newcommand{\proj}{\ensuremath{\text{\rm proj}}}

\newcommand{\spa}{\ensuremath{\text{span}\,}}

\newcommand{\sign}{\ensuremath{\text{\rm sign}}}

\newcommand{\inte}{\ensuremath{\text{\rm int}}}

\newcommand{\cone}{\ensuremath{\text{\rm cone}\,}}

\newcommand{\reli}{\ensuremath{\text{\rm ri}\,}}
\newcommand{\infconv}{\ensuremath{\mbox{\small$\,\square\,$}}}

\newtheorem{theorem}{Theorem}[section]
\newtheorem{lemma}[theorem]{Lemma}

\newtheorem{proposition}[theorem]{Proposition}

\theoremstyle{plain}{\theorembodyfont{\rmfamily}%
}
\theoremstyle{plain}{\theorembodyfont{\rmfamily}%
\newtheorem{example}[theorem]{Example}}
\theoremstyle{plain}{\theorembodyfont{\rmfamily}%
\newtheorem{remark}[theorem]{Remark}}
\theoremstyle{plain}{\theorembodyfont{\rmfamily}%
}
\theoremstyle{plain}{\theorembodyfont{\rmfamily}%
}
\theoremstyle{plain}{\theorembodyfont{\rmfamily}%
}
\theoremstyle{plain}{\theorembodyfont{\rmfamily}%
}
\theoremstyle{plain}{\theorembodyfont{\rmfamily}%
\newtheorem{model}[theorem]{Model}}
\theoremstyle{plain}{\theorembodyfont{\rmfamily}%
\newtheorem{problem}[theorem]{Problem}}

\numberwithin{equation}{section}
\begin{document}

\title{\sffamily\LARGE Perspective Maximum Likelihood-Type 
Estimation\\ via Proximal Decomposition\footnote{Contact author: 
P. L. Combettes, \email{plc@math.ncsu.edu},
phone: +1 (919) 515 2671. The work of P. L. Combettes was 
supported by the National Science Foundation under grant 
DMS-1818946.}}

\author{Patrick L. Combettes$^1$ and Christian L. M\"uller$^2$\\
\small
\small $\!^1$North Carolina State University,
Department of Mathematics,
Raleigh, NC 27695-8205, USA\\
\small \email{plc@math.ncsu.edu}\\
\small \vskip 2mm
\small $\!^2$Flatiron Institute, Simons Foundation,
New York, NY 10010, USA\\
\small \email{cmueller@flatironinstitute.org}
}

\date{\sffamily ~}

\maketitle

\vskip 8mm

\begin{abstract}
We introduce an optimization model for maximum likelihood-type
estimation (M-estimation) that generalizes a large class of
existing statistical models, including Huber's concomitant
M-estimator, Owen's Huber/Berhu concomitant estimator, 
the scaled lasso, support vector machine
regression, and penalized estimation with structured sparsity.
The model, termed perspective M-estimation, leverages the
observation that convex M-estimators with concomitant scale as
well as various regularizers are instances of perspective functions. 
Such functions are amenable to proximal analysis, which
leads to principled and provably convergent optimization algorithms 
via proximal splitting. Using a geometrical approach based on 
duality, we derive novel proximity operators for several
perspective functions of interest. Numerical experiments on
synthetic and real-world data illustrate the broad applicability of
the proposed framework. 
\end{abstract}

\newpage
\section{Introduction}
High-dimensional regression methods play a pivotal role in 
modern data analysis. A large body of statistical work has
focused on estimating regression coefficients under various
structural assumptions, such as sparsity of the regression vector
\cite{Tibs96}. In the standard linear model, regression
coefficients constitute, however, only one aspect of the model. 
A more
fundamental objective in statistical inference is the estimation of
both location (i.e., the regression coefficients) and scale (e.g.,
the standard deviation of the noise) of the statistical model from
the data. A common approach is to decouple this estimation process by
designing and analyzing individual estimators for scale and
location parameters (see, e.g., \cite[pp.~140]{Hube81},
\cite{Yuch17}) because joint estimation often leads to non-convex
formulations \cite{Daye2012,Stad10}. One important exception has
been proposed in robust statistics in the form of a maximum
likelihood-type estimator (M-estimator) for location with
concomitant scale \cite[pp.~179]{Hube81}, which couples
both parameters via a convex objective function. To discuss this
approach more precisely, we introduce the linear
heteroscedastic mean shift regression model. This data
formation model will be used throughout the paper.
 
\begin{model}
\label{m:1}
The vector $y=(\eta_i)_{1\leq i\leq n}\in\RR^n$ of observations is 
\begin{equation}
\label{e:1}
y=X\overline{b}+\overline{o}+Ce,
\end{equation}
where $X\in\RR^{n\times p}$ is a known design matrix with rows
$(x_i)_{1\leq i\leq n}$, $\overline{b}\in\RR^p$ is the unknown
regression vector (location), $\overline{o}\in\RR^n$ is the unknown
mean shift
vector containing outliers, $e\in\RR^n$ is a vector of realizations
of i.i.d. zero mean random variables, and $C\in\RP^{n\times n}$ is
a diagonal matrix the diagonal of which are the (unknown)
non-negative standard deviations. One obtains the 
homoscedastic mean shift model when the diagonal entries 
of $C$ are identical.
\end{model}

The concomitant M-estimator proposed in \cite[pp.~179]{Hube81} is
based on the objective function 
\begin{equation}
\label{e:HuberOrig}
(\sigma,b)\mapsto {\frac{\sigma}{n}\sum_{i=1}^n 
\Bigg(\mathsf{h}_{\rho_1} 
\bigg(\frac{x_i^\top b-\eta_i}{\sigma}\bigg)+\delta\Bigg)},
\end{equation}
where $\mathsf{h}_{\rho_1}$ is the Huber function \cite{Hube64}
with parameter $\rho_1\in\RPP$, $\delta\in\RP$, and the scalar
$\sigma$ is a scale.
The objective function, which we also refer to as the homoscedastic 
Huber M-estimator function, is jointly convex in both $b$
and scalar $\sigma$, and hence, amenable to global optimization.
Under suitable assumptions, this estimator can identify outliers $o$ 
and can estimate a scale that is proportional to the diagonal 
entries of $C$ in the homoscedastic case,
In \cite{Anto10}, it 
was proposed that joint convex optimization of
regression vector and standard deviation may also be advantageous
in sparse linear regression. There, the objective function is
\begin{equation}
\label{e:AntonOrig}
(\sigma,b)\mapsto {\frac{\sigma}{n}\sum_{i=1}^n \Bigg(
\Bigg|\frac{x_i^\top b-\eta_i}{\sigma}\Bigg|^2+
\delta\Bigg)+\alpha_1\|b\|_1},
\end{equation}
where the term $\|\cdot\|_1$ promotes sparsity of the regression
estimate, $\alpha_1\in\RPP$ is a tuning parameter, and $\sigma$ is an
estimate of the standard deviation. This objective function is at
the heart of the scaled lasso estimator \cite{Sunt12}. The
resulting estimator
is not robust to outliers but is equivariant, which makes
the tuning parameter $\alpha_1$ independent of the noise level. In
\cite{Owen07}, an extension of \eqref{e:HuberOrig} was introduced
that includes a new penalization function as well as 
concomitant scale
estimation for the regression vector. The 
objective function is
\begin{equation}
\label{e:OwenOrig}
(\sigma,\tau,b)\mapsto {\frac{\sigma}{n}\sum_{i=1}^n
\Bigg(\mathsf{h}_{\rho_1} 
\bigg(\frac{x_i^\top b-\eta_i}{\sigma}\bigg)+\delta_1\Bigg) +
\frac{\alpha_1\tau}{p}\sum_{i=1}^p \Bigg( \mathsf{b}_{\rho_2} 
\bigg(\frac{\beta_i}{\tau}\bigg)+\delta_2\Bigg)},
\end{equation}
where $\mathsf{b}_{\rho_2}$ is the reverse Huber (Berhu) function
\cite{Owen07} with parameter $\rho_2\in\RPP$, constants
$\delta_1\in\RPP$ and $\delta_2\in\RPP$, and tuning parameter
$\alpha_1\in\RPP$. This objective function is jointly
convex in $b$ and the scalar parameters $\sigma$ and $\tau$. The
estimator
inherits the equivariance and robustness of the previous
estimators. In addition, the Berhu penalty is advantageous when the
design matrix comprises correlated rows \cite{Lamb16}.   
In \cite{Jmaa18}, it was observed that these objective functions, as
well as many regularization penalties for structured sparsity 
\cite{Bach2011,Micc13,Macd16}, are instances of the class of 
composite ``perspective functions'' \cite{Svva17}. 

In the present paper, we leverage the ubiquity of perspective
functions in statistical M-estimation and introduce a new
statistical optimization model, \emph{perspective M-estimation}. 
The perspective M-estimation model, put forward in detail in
\eqref{e:prob0}, uses perspective functions as
fundamental building blocks to couple scale and regression
variables in a jointly convex fashion. It includes in particular the 
M-estimators discussed above as special cases. For a large class of
perspective
functions, proximal analysis enables the principled construction of
proximity operators, a key ingredient for minimization of the model
using proximal algorithms \cite{Jmaa18}. Using geometrical insights
revealed by the dual problem, we derive new proximity
operators for several perspective functions, including
the generalized scaled lasso, the generalized Huber, the abstract
Vapnik, and the generalized Berhu function. This enables the
development of a unifying algorithmic framework for global
optimization of the proposed model using modern splitting
techniques. The model also allows seamless integration of a large
class of regularizers for structured sparsity and novel robust
heteroscedastic estimators of location and scale. Numerical
experiments on synthetic and real-world data illustrate the
applicability of the framework.

\section{Proximity operators of perspective functions}
\label{sec:2}

The general perspective M-estimation model proposed in
Problem~\ref{prob:1} hinges on the notion of a perspective function
(see \eqref{e:perspective} below). 
To solve this problem we need
to be able to compute the proximity operators of such functions.
The properties of these proximity operators were investigated in 
\cite{Jmaa18}, where some examples of computation were presented.
In this section, we derive further instances of explicit
expressions for the proximity operator of perspective functions.
Since these results are of general interest beyond
statistical analysis, throughout, $\HH$ is a real
Hilbert space with scalar product $\scal{\cdot}{\cdot}$ and
associated norm $\|\cdot\|$.

\subsection{Notation and background on convex analysis}

The closed ball with center $x\in\HH$ and radius $\rho\in\RPP$ is
denoted by $B(x;\rho)$. Let $C$ be a subset of $\HH$. Then 
\begin{equation}
\label{e:iota}
\iota_C\colon\HH\to\{0,\pinf\}\colon x\mapsto
\begin{cases}
0,&\text{if}\;\;x\in C;\\
\pinf,&\text{if}\;\;x\notin C
\end{cases}
\end{equation}
is the indicator function of $C$, 
\begin{equation}
\label{e:dictC}
d_C\colon\HH\to\RPX\colon x\mapsto\inf_{y\in C}\|y-x\|
\end{equation}
is the distance function to $C$, and 
\begin{equation}
\label{e:support}
\sigma_C\colon\HH\to\RXX\colon u
\mapsto\sup_{x\in C}\scal{x}{u}
\end{equation}
is the support function of $C$. 
If $C$ is nonempty, closed, and convex then, for every 
$x\in\HH$, there exists a unique point $\proj_Cx\in C$, called 
the projection of $x$ onto $C$, such that $\|x-\proj_Cx\|=d_C(x)$.
We have
\begin{equation}
\label{e:kolmogorov}
(\forall x\in\HH)(\forall p\in\HH)\quad
p=\proj_Cx\quad\Leftrightarrow\quad\big[\,p\in C
\quad\text{and}\quad
(\forall y\in C)\quad\scal{y-p}{x-p}\leq 0\,\big].
\end{equation}
The normal cone to $C$ is 
\begin{equation}
\label{e:nc}
N_C=\partial\iota_C\colon\HH\to 2^\HH\colon x\mapsto 
\begin{cases}
\menge{u\in\HH}{\sup\scal{C-x}{u}\leq 0},
&\text{if}\;\;x\in C;\\
\emp,&\text{otherwise.}
\end{cases}
\end{equation}
A function $\varphi\colon\HH\to\RX$ is proper if 
$\dom\varphi=\menge{x\in\HH}{\varphi(x)<\pinf}\neq\emp$ and
coercive if $\lim_{\|x\|\to\pinf}\varphi(x)=\pinf$.
We denote by $\Gamma_0(\HH)$ the class of proper lower 
semicontinuous convex functions from $\HH$ to $\RX$. Let
$\varphi\in\Gamma_0(\HH)$. The conjugate of $\varphi$ is 
\begin{equation}
\label{e:conj}
\varphi^*\colon\HH\to\RX\colon u\mapsto\sup_{x\in\HH}\big(
\scal{x}{u}-\varphi(x)\big).
\end{equation}
It also belongs to $\Gamma_0(\HH)$ and $\varphi^{**}=\varphi$.
The Moreau subdifferential of $\varphi$ is the set-valued operator
\begin{equation}
\label{e:subdiff}
\partial \varphi\colon\HH\to 2^\HH\colon x\mapsto 
\menge{u\in\HH}{(\forall y\in\dom \varphi)\;\:
\scal{y-x}{u}+\varphi(x)\leq \varphi(y)}.
\end{equation}
We have
\begin{equation}
\label{e:subdiff2}
(\forall x\in\HH)(\forall u\in\HH)\quad
u\in\partial \varphi(x)\quad\Leftrightarrow\quad
x\in\partial \varphi^*(u).
\end{equation}
Moreover,
\begin{equation}
\label{e:subdiff11}
(\forall x\in\HH)(\forall u\in\HH)\quad
\varphi(x)+\varphi^*(u)\geq\scal{x}{u}
\end{equation}
and 
\begin{equation}
\label{e:subdiff12}
(\forall x\in\HH)(\forall u\in\HH)\quad
u\in\partial \varphi(x)\quad\Leftrightarrow\quad
\varphi(x)+\varphi^*(u)=\scal{x}{u}.
\end{equation}
If $\varphi$ is G\^ateaux differentiable at 
$x\in\dom f$, with gradient $\nabla \varphi(x)$, then 
\begin{equation}
\label{e:2384572k}
\partial\varphi(x)=\{\nabla \varphi(x)\}. 
\end{equation}
The infimal convolution of $\varphi$ and $\psi\in\Gamma_0(\HH)$ is
\begin{equation}
\label{e:infconv}
\varphi\infconv\psi\colon\HH\to\RXX\colon x\mapsto
\inf_{y\in\HH}\big(\varphi(y)+\psi(x-y)\big).
\end{equation}
Given any $z\in\dom\varphi$, the recession function of $\varphi$ is
\begin{equation}
\label{e:2015-09-23a}
(\forall x\in\HH)\quad(\rec\varphi)(x)=\sup_{y\in\dom\varphi}
\big(\varphi(x+y)-\varphi(y)\big)=\lim_{\alpha\to\pinf}
\frac{\varphi(z+\alpha x)}{\alpha}.
\end{equation}
Finally, the proximity operator of $\varphi$ is \cite{Mor62b}
\begin{equation}
\label{e:jjm1962}
\prox_\varphi\colon\HH\to\HH\colon x\mapsto
\underset{y\in\HH}{\text{argmin}}
\bigg(\varphi(y)+\frac{1}{2}\|x-y\|^2\bigg).
\end{equation}
For detailed accounts of convex analysis, see 
\cite{Livre1,Rock70}.

\subsection{Perspective functions}
Let $\varphi\in\Gamma_0(\HH)$. The perspective of $\varphi$ is
\begin{equation}
\label{e:perspective}
\widetilde{\varphi}\colon\RR\times\HH\to\RX\colon 
(\sigma,x)\mapsto
\begin{cases}
\sigma\varphi(x/\sigma),&\text{if}\;\:\sigma>0;\\
(\rec\varphi)(x),&\text{if}\;\:\sigma=0;\\
\pinf,&\text{otherwise.}
\end{cases}
\end{equation}
We have $\widetilde{\varphi}\in\Gamma_0(\RR\oplus\GG)$
\cite[Proposition~2.3]{Svva17}. The following result is useful
to establish existence results for problems involving perspective
functions.

\begin{proposition}
\label{p:coer}
Let $\varphi\in\Gamma_0(\HH)$ be such that $\inf\varphi(\HH)>0$ and
$0\in\intdom\varphi^*$. Then $\widetilde{\varphi}$ is coercive.
\end{proposition}
\begin{proof}
We have $\varphi^*(0)=-\inf\varphi(\HH)<0$ and
$0\in\intdom\varphi^*$. Hence,
$(0,0)\in\inte\,\epi{\varphi^*}$.
In turn, we derive from \cite[Proposition~2.3(iv)]{Svva17} that 
\begin{equation}
(0,0)\in\inte\,\menge{(\mu,u)\in\RR\oplus\HH}
{\mu+\varphi^*(u)\leq 0}=
\inte\,\dom(\widetilde{\varphi})^*.
\end{equation}
It therefore follows from \cite[Proposition~14.16]{Livre1} that 
$\widetilde{\varphi}$ is coercive.
\end{proof}

Let us now turn to the proximity operator of
$\widetilde{\varphi}$.

\begin{lemma}{\rm\cite[Theorem~3.1]{Jmaa18}}
\label{l:1}
Let $\varphi\in\Gamma_0(\HH)$, let $\gamma\in\RPP$, let
$\sigma\in\RR$, and let $x\in\HH$. Then the following hold:
\begin{enumerate}
\item
\label{l:1i}
Suppose that $\sigma+\gamma\varphi^*(x/\gamma)\leq 0$. 
Then $\prox_{\gamma\widetilde{\varphi}}(\sigma,x)=(0,0)$.
\item
\label{l:1ii}
Suppose that $\dom\varphi^*$ is open and that 
$\sigma+\gamma\varphi^*(x/\gamma)>0$. Then 
\begin{equation}
\label{e:ramm1}
\prox_{\gamma\widetilde{\varphi}}(\sigma,x)=
\big(\sigma+\gamma\varphi^*(p),x-\gamma p\big), 
\end{equation}
where $p$ is the unique solution to the inclusion
$x\in\gamma p+(\sigma+\gamma\varphi^*(p))\partial\varphi^*(p)$.
If $\varphi^*$ is differentiable at $p$, then
$p$ is characterized by
$x=\gamma p+(\sigma+\gamma\varphi^*(p))\nabla\varphi^*(p)$.
\end{enumerate}
\end{lemma}

When $\dom\varphi^*$ is not open, we propose a
geometric construction instead of Lemma~\ref{l:1} to compute 
$\prox_{\gamma\widetilde{\varphi}}$ via the projection onto a
certain convex set. It is based on the following property, which 
reduces the problem of evaluating the proximity operator of
$\widetilde{\varphi}$ to a projection problem in $\RR^2$ if
$\varphi$ is radially symmetric.

\begin{proposition}
\label{p:1}
Let $\phi\in\Gamma_0(\RR)$ be an even function, set 
$\varphi=\phi\circ\|\cdot\|\colon\HH\to\RX$, let 
$\gamma\in\RPP$, let $\sigma\in\RR$, and let $x\in\HH$. Set
\begin{equation}
\label{e:R}
\EuScript{R}=\menge{(\chi,\nu)\in\RR^2}{\chi+\phi^*(\nu)\leq 0}.
\end{equation}
Then $\EuScript{R}$ is a nonempty closed convex set, and
the following hold:
\begin{enumerate}
\item
\label{p:1i}
Suppose that $\sigma+\gamma\phi^*(\|x\|/\gamma)\leq 0$. Then 
$\prox_{\gamma\widetilde{\varphi}}(\sigma,x)=(0,0)$.
\item
\label{p:1ii}
Suppose that $\sigma>\gamma\phi(0)$ and $x=0$. 
Then 
$\prox_{\gamma\widetilde{\varphi}}(\sigma,x)=
(\sigma-\gamma\phi(0),x)$.
\item
\label{p:1iii}
Suppose that $\sigma+\gamma\phi^*(\|x\|/\gamma)>0$ and $x\neq 0$, 
and set 
$(\chi,\nu)=\proj_{\EuScript{R}}(\sigma/\gamma,\|x\|/\gamma)$. Then 
\begin{equation}
\label{e:duhast0}
\prox_{\gamma\widetilde{\varphi}}(\sigma,x)=
\Bigg(\sigma-\gamma\chi,\bigg(1-\dfrac{\gamma\nu}{\|x\|}\bigg)
\,x\Bigg).
\end{equation}
\end{enumerate}
\end{proposition}
\begin{proof}
The properties of $\EuScript{R}$ follow from the fact that
$\phi^*\in\Gamma_0(\RR)$. Now, let us recall from 
\cite[Remark~3.2]{Jmaa18} that
\begin{equation}
\label{e:C}
\EuScript{C}=
\menge{(\mu,u)\in\RR\oplus\HH}{\mu+\varphi^*(u)\leq 0}
\end{equation}
and that
\begin{equation}
\label{e:duhast}
\prox_{\gamma\widetilde{\varphi}}(\sigma,x)=(\sigma,x)-
\gamma\proj_{\EuScript{C}}\big(\sigma/\gamma,x/\gamma\big).
\end{equation}
In addition, \cite[Example~13.8]{Livre1} states that 
\begin{equation}
\label{e:C2}
\varphi^*=\phi^*\circ\|\cdot\|.
\end{equation}

\ref{p:1i}: 
This follows from \eqref{e:C2} and Lemma~\ref{l:1}\ref{l:1i}.

\ref{p:1ii}:
Let us show that
$\proj_{\EuScript{C}}(\sigma/\gamma,0)=
(\phi(0),0)$, which will establish the claim by virtue
of \eqref{e:duhast}. Since $\phi$ is an even function in 
$\Gamma_0(\RR)$, 
$\phi(0)=\inf\phi(\RR)=-\phi^*(0)$. Hence
$\phi(0)+\varphi^*(0)=\phi(0)+\phi^*(0)=0$ and 
$(\phi(0),0)\in\EuScript{C}$. Now fix 
$(\eta,y)\in\EuScript{C}$. Then, since $\varphi$ is an even
function in $\Gamma_0(\HH)$,
$\eta\leq-\varphi^*(y)\leq-\inf\varphi^*(\HH)
=-\varphi^*(0)=-\phi^*(0)=\phi(0)$ and, 
since $\sigma>\gamma\phi(0)$, we get 
\begin{equation}
\scal{(\eta,y)-(\phi(0),0)}{(\sigma/\gamma,0)-(\phi(0),0)}=
\big(\eta-\phi(0)\big)\big(\sigma/\gamma-\phi(0)\big)\leq 0.
\end{equation}
Altogether, \eqref{e:kolmogorov} asserts that 
$\proj_{\EuScript{C}}(\sigma/\gamma,0)=(\phi(0),0)$.

\ref{p:1iii}: 
In view of \eqref{e:duhast}, it is enough to show that
$\proj_{\EuScript{C}}(\sigma/\gamma,x/\gamma)=
(\chi,\nu x/\|x\|)$.
Since $(\chi,\nu)\in\EuScript{R}$, \eqref{e:C2} yields 
$\chi+\varphi^*(\nu x/\|x\|)=\chi+\phi^*(\nu)\leq 0$ and, 
therefore, $(\chi,\nu x/\|x\|)\in\EuScript{C}$. On the other hand,
we infer from \eqref{e:C2} that $\EuScript{C}\subset\RR\oplus\HH$ 
is radially symmetric in the $\HH$-direction. As a result, 
$\proj_{\EuScript{C}}\big(\sigma/\gamma,x/\gamma\big)\in
V=\RR\times\spa\{x\}$ and therefore 
$\proj_{\EuScript{C}}\big(\sigma/\gamma,x/\gamma\big)=
\proj_{V\cap\EuScript{C}}\big(\sigma/\gamma,x/\gamma\big)$
\cite[Proposition~29.5]{Livre1}.
Now fix $(\eta,y)\in V\cap\EuScript{C}$.
Then $(\eta,\pm\|y\|)\in\EuScript{R}$ and \eqref{e:kolmogorov}
yields
\begin{equation}
(\eta-\chi)(\sigma/\gamma-\chi)+
(\pm\|y\|-\nu)(\|x\|/\gamma-\nu)=\scal{(\eta-\chi,\pm\|y\|-\nu)}
{(\sigma/\gamma-\chi,\|x\|/\gamma-\nu)}_{\RR^2}\leq 0.
\end{equation}
Hence, since $y=\pm\|y\|x/\|x\|$, 
\begin{align}
&\hskip -4mm
\scal{(\eta,y)-(\chi,\nu x/\|x\|)}
{(\sigma/\gamma,x/\gamma)-(\chi,\nu x/\|x\|)}_{\RR\oplus\HH}
\nonumber\\
&=(\eta-\chi)(\sigma/\gamma-\chi)+
\scal{y-\nu x/\|x\|}{x/\gamma-\nu x/\|x\|}\nonumber\\
&=(\eta-\chi)(\sigma/\gamma-\chi)+
\scal{\pm\|y\|x/\|x\|-\nu x/\|x\|}{\|x\|x/(\gamma\|x\|)-
\nu x/\|x\|}\nonumber\\
&=(\eta-\chi)(\sigma/\gamma-\chi)+
(\pm\|y\|-\nu)(\|x\|/\gamma-\nu)\scal{x}{x}/\|x\|^2\nonumber\\
&=(\eta-\chi)(\sigma/\gamma-\chi)+
(\pm\|y\|-\nu)(\|x\|/\gamma-\nu)\nonumber\\
&\leq 0.
\end{align}
Altogether, we derive from \eqref{e:kolmogorov} that 
$(\chi,\nu x/\|x\|)
=\proj_{V\cap\EuScript{C}}\big(\sigma/\gamma,x/\gamma\big)=
\proj_{\EuScript{C}}\big(\sigma/\gamma,x/\gamma\big)$.
\end{proof}

\begin{figure}
\centering
\scalebox{0.50} 
{
\hskip -1mm
\begin{pspicture}(-5,-2)(5,6.8) 
\psline[linewidth=0.04cm,arrowsize=0.2cm,linestyle=solid]{->}%
(-4,0)(3.9,0)
\psline[linewidth=0.04cm,arrowsize=0.2cm,linestyle=solid]{->}%
(0,-1.8)(0,5.8)
\psplot[plotpoints=800,algebraic,linewidth=0.06cm,linecolor=blue]%
{-1.0}{1.0}{1+abs(x)}
\psplot[plotpoints=800,algebraic,linewidth=0.06cm,linecolor=blue]%
{1.0}{2.9}{1+(x^2+1)/2}
\psplot[plotpoints=800,algebraic,linewidth=0.06cm,linecolor=blue]%
{-2.9}{-1.0}{1+(x^2+1)/2}
\rput(4.8,0){\LARGE$x\in\HH$}
\rput(0,6.3){\LARGE$\varphi(x)$}
\rput(1.0,-0.45){\LARGE$\rho$}
\rput(-1.0,-0.0){\small${|}$}
\rput(1.0,-0.0){\small${|}$}
\rput(0.0,1.0){\LARGE${-}$}
\rput(0.3,0.8){\LARGE$\alpha$}
\rput(-1.0,-0.45){\LARGE$-\rho$}
\end{pspicture} 
\hskip 8mm
\begin{pspicture}(-5,-2)(5,6.8) 
\psline[linewidth=0.04cm,arrowsize=0.2cm,linestyle=solid]{->}%
(-4,0)(3.9,0)
\psline[linewidth=0.04cm,arrowsize=0.2cm,linestyle=solid]{->}%
(0,-1.8)(0,5.8)
\psplot[plotpoints=800,algebraic,linewidth=0.06cm,linecolor=red]%
{-1.0}{1.0}{-1}
\psplot[plotpoints=800,algebraic,linewidth=0.06cm,linecolor=red]%
{1.0}{3.8}{-1+abs(x-1)+(x-1)^2/2}
\psplot[plotpoints=800,algebraic,linewidth=0.06cm,linecolor=red]%
{-3.8}{-1.0}{-1+abs(x+1)+(x+1)^2/2}
\rput(4.8,0){\LARGE${u}\in\HH$}
\rput(0,6.3){\LARGE$\varphi^*(u)$}
\rput(1.0,-0.45){\LARGE$\kappa$}
\rput(-1.0,-0.0){\small${|}$}
\rput(1.0,-0.0){\small${|}$}
\rput(0.5,-1.3){\LARGE$-\alpha$}
\rput(-1.0,-0.45){\LARGE${-\kappa}$}
\rput(0.0,-1.0){\small${-}$}
\end{pspicture} 
\hskip 8mm
\begin{pspicture}(-5,-5)(6,5.1) 
\psline[linewidth=0.04cm,arrowsize=0.2cm,%
linecolor=dbrown](3.5,-2.5)(1,-1.0)
\psline[linewidth=0.04cm,arrowsize=0.2cm,linecolor=dbrown]%
(3.5,-1.0)(1,-1.0)
\psline[linewidth=0.04cm,arrowsize=0.2cm,%
linecolor=dbrown](3.5,2.5)(1,1.0)
\psline[linewidth=0.04cm,arrowsize=0.2cm,linecolor=dbrown]%
(3.5,1.0)(1,1.0)
\psline[linewidth=0.04cm,arrowsize=0.2cm,linecolor=dbrown]{->}%
(3.5,0.67)(1.0,0.67)
\psline[linewidth=0.04cm,arrowsize=0.2cm,linecolor=dbrown]{->}%
(3.5,0.33)(1.0,0.33)
\psline[linewidth=0.04cm,arrowsize=0.2cm,linecolor=dbrown]{->}%
(3.5,-0.67)(1.0,-0.67)
\psline[linewidth=0.04cm,arrowsize=0.2cm,linecolor=dbrown]{->}%
(3.5,-0.33)(1.0,-0.33)
\psline[linewidth=0.04cm,arrowsize=0.2cm,linestyle=solid]{->}%
(-4,0)(4.2,0)
\psline[linewidth=0.04cm,arrowsize=0.2cm,linestyle=solid]{->}%
(0,-4.5)(0,4.8)
\psline[linewidth=0.04cm,arrowsize=0.2cm,linecolor=dbrown]{->}%
(3.5,1.5)(1,1.0)
\psline[linewidth=0.04cm,arrowsize=0.2cm,linecolor=dbrown]{->}%
(3.5,2.0)(1,1.0)
\psline[linewidth=0.04cm,arrowsize=0.2cm,linecolor=dbrown]{->}%
(3.5,-1.5)(1,-1.0)
\psline[linewidth=0.04cm,arrowsize=0.2cm,linecolor=dbrown]{->}%
(3.5,-2.0)(1,-1.0)
\psline[linewidth=0.04cm,arrowsize=0.2cm,linecolor=dbrown,%
linestyle=solid]{->}(0.6,-4.8)(-0.6,-2.08)
\psline[linewidth=0.04cm,arrowsize=0.2cm,linecolor=dbrown,%
linestyle=solid]{->}(0.55,4.6)(-0.6,2.08)
\psline[linewidth=0.04cm,arrowsize=0.2cm,linecolor=dbrown,%
linestyle=solid]{->}(-1.1,-5.0)(-2.1,-2.71)
\psline[linewidth=0.04cm,arrowsize=0.2cm,linecolor=dbrown,%
linestyle=solid]{->}(-1.1,5.0)(-2.1,2.71)
\psline[linewidth=0.04cm,arrowsize=0.2cm,linecolor=dbrown,%
linestyle=solid]{->}(-2.94,-5.2)(-3.65,-3.23)
\psline[linewidth=0.04cm,arrowsize=0.2cm,linecolor=dbrown,%
linestyle=solid]{->}(-2.94,5.2)(-3.65,3.23)
\psline[linewidth=0.04cm,arrowsize=0.2cm,linecolor=dbrown,%
linestyle=solid]{->}(3.0,3.8)(0.8,1.2)
\psline[linewidth=0.04cm,arrowsize=0.2cm,linecolor=dbrown,%
linestyle=solid]{->}(3.0,-3.8)(0.8,-1.2)
\psline[linewidth=0.06cm,linecolor=red](1,-1)(1,1)
\psplot[plotpoints=800,algebraic,linewidth=0.06cm,linecolor=red]%
{-4.0}{1.0}{(-2*(x-1)+1)^(1/2)}
\psplot[plotpoints=800,algebraic,linewidth=0.06cm,linecolor=red]%
{-4.0}{1.0}{-(-2*(x-1)+1)^(1/2)}
\rput(5.0,0){\LARGE$\mu\in\RR$}
\rput(0,5.15){\LARGE$u\in\HH$}
\rput(0.7,0.3){\LARGE$\alpha$}
\rput(1.0,-0.0){\small${|}$}
\rput(0.0,1.0){\large${-}$}
\rput(0.0,-1.0){\large${-}$}
\rput(-0.4,1.0){\LARGE$\kappa$}
\rput(-0.59,-1.0){\LARGE${-\kappa}$}
\rput(-2.3,1.4){\red\LARGE$\EuScript{C}$}
\end{pspicture} 
}
\caption{Geometry of the computation of
$\prox_{\widetilde{\varphi}}$ in \eqref{e:duhast}.
Left: original function $\varphi$.
Center: conjugate of $\varphi$.
Right: action of the projection operator $\proj_{\EuScript{C}}$
onto the set ${\EuScript{C}}$ of \eqref{e:C}. The proximity 
operator of ${\widetilde{\varphi}}$ is 
$\Id-\proj_{\EuScript{C}}$. In the specific
example depicted here, $\HH=\RR$ and $\varphi$ is the 
Berhu function of \eqref{e:owen}.}
\label{fig:1}
\end{figure}
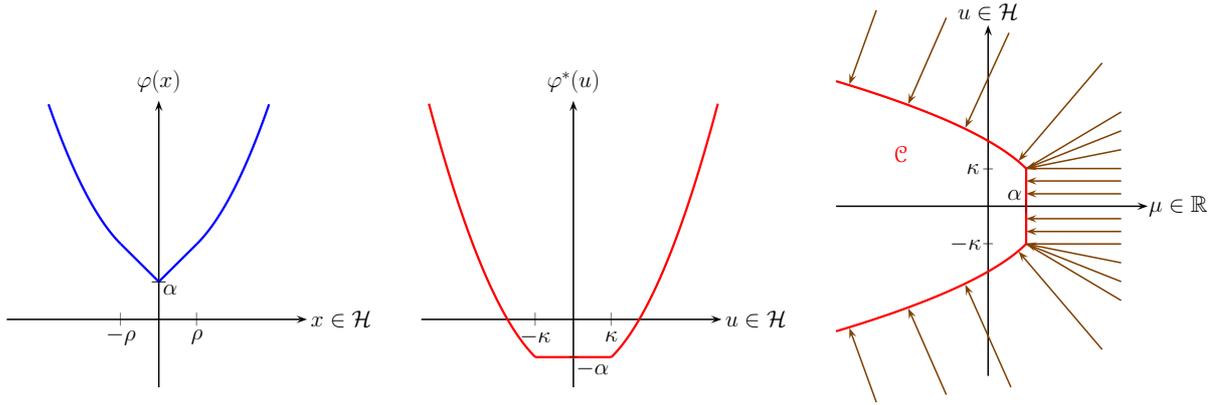

\subsection{Examples}

We provide several examples that are relevant to the statistical
problems we have in sight. 

\begin{example}[generalized scaled lasso function]
{\rm \cite[Example~3.7]{Jmaa18}}
\label{ex:sl}
Let $\alpha\in\RPP$, $\gamma\in\RPP$, $\kappa\in\RPP$, 
$q\in\left]1,\pinf\right[$, $\sigma\in\RR$, and $x\in\HH$. Set 
$\varphi=\alpha+\|\cdot\|^q/\kappa\colon\HH\to\RR$ and 
$q^*=q/(q-1)$. Then 
\begin{equation}
\label{e:jmaa18-2}
\widetilde{\varphi}(\sigma,x)=
\begin{cases}
\alpha\sigma+\Frac{\|x\|^q}{\kappa\sigma^{q-1}},
&\text{if}\;\;\sigma>0; \\[3mm]
0,&\text{if}\;\;x=0\;\;\text{and}\;\;\sigma=0;\\
\pinf, &\text{otherwise.}
\end{cases}
\end{equation}
Now set $\rho=(\kappa/q)^{q^*-1}$. If
$q^*\gamma^{q^*-1}\sigma+\rho\|x\|^{q^*}>q^*\gamma^{q^*}\alpha$ 
and $x\neq 0$, let $t$ be the unique 
solution in $\RPP$ to the equation
\begin{equation}
\label{e:nyc2016-06-10f}
t^{2q^*-1}+\frac{q^*(\sigma-\gamma\alpha)}{\gamma\rho}t^{q^*-1}
+\frac{q^*}{\rho^2}t-\frac{q^*\|x\|}{\gamma\rho^2}=0.
\end{equation}
Set $p=tx/\|x\|$ if $x\neq 0$, and $p=0$ if $x=0$.
Then (note \cite[Eq.~(3.47)]{Jmaa18} is incorrect when
$\alpha\neq 0$)
\begin{equation}
\label{e:nyc2016-06-10b}
\prox_{\gamma\widetilde{\varphi}}(\sigma,x)=
\begin{cases}
\Big(\sigma+\gamma\big(\rho t^{q^*}/q^*-\alpha\big),
x-\gamma p\Big),
&\text{if}\;\;q^*\gamma^{q^*-1}\sigma+\rho\|x\|^{q^*}
>q^*\gamma^{q^*}\alpha;\\
\big(0,0\big),&\text{if}\;\;q^*\gamma^{q^*-1}\sigma+
\rho\|x\|^{q^*}\leq q^*\gamma^{q^*}\alpha.
\end{cases}
\end{equation}
\end{example}

Given $\rho\in\RPP$, the classical Huber function is defined as
\cite{Hube64}
\begin{equation}
\label{e:huber}
\mathsf{h}_\rho\colon\RR\to\RR\colon \xi\mapsto
\begin{cases}
\rho|\xi|-\Frac{\rho^{2}}{2},
&\text{if}\;\;|\xi|>\rho;\\[4mm]
\Frac{|\xi|^2}{2},&\text{if}\;\;|\xi|\leq\rho,
\end{cases}
\end{equation}
and it is known as the Huber function. Below, we study the
perspective of a generalization of it.

\begin{example}[generalized Huber function]
\label{ex:gh}
Let $\alpha$, $\gamma$, and $\rho$ be in $\RPP$, let 
$q\in\left]1,\pinf\right[$, and set $q^*=q/(q-1)$. Define
\begin{equation}
\label{e:gh}
\varphi\colon\HH\to\RR\colon x\mapsto
\begin{cases}
\alpha-\Frac{\rho^{q^*}}{q^*}+\rho\|x\|,
&\text{if}\;\;\|x\|>\rho^{q^*/q};\\[4mm]
\alpha+\Frac{\|x\|^q}{q},&\text{if}\;\;\|x\|\leq\rho^{q^*/q}.
\end{cases}
\end{equation}
Let $\sigma\in\RR$ and $x\in\HH$. Then 
\begin{equation}
\label{e:hb1}
\widetilde{\varphi}(\sigma,x)=
\begin{cases}
\bigg(\alpha-\Frac{\rho^{q^*}}{q^*}\bigg)
\sigma+\rho\|x\|,&\text{if}\;\;\sigma>0\;\;\text{and}\;\;
\|x\|>\sigma\rho^{q^*/q};\\[4mm]
\alpha\sigma+\dfrac{\|x\|^q}{q\sigma^{q-1}},
&\text{if}\;\;\sigma>0\;\;\text{and}\;\;
\|x\|\leq\sigma\rho^{q^*/q};\\[3mm]
\rho\|x\|,&\text{if}\;\;\sigma=0;\\
\pinf,&\text{if}\;\;\sigma<0.
\end{cases}
\end{equation}
In addition, the following hold:
\begin{enumerate}
\item
\label{ex:ghi}
Suppose that $\|x\|\leq\gamma\rho$ and 
$\|x\|^{q^*}\leq\gamma^{q^*}q^*(\alpha-\sigma/\gamma)$. Then 
$\prox_{\gamma\widetilde{\varphi}}(\sigma,x)=(0,0)$.
\item
\label{ex:ghii}
Suppose that $\sigma\leq\gamma(\alpha-\rho^{q^*}/q^*)$ and 
$\|x\|>\gamma\rho$. Then 
\begin{equation}
\label{e:gh3}
\prox_{\gamma\widetilde{\varphi}}(\sigma,x)=
\Bigg(0,\bigg(1-\dfrac{\gamma\rho}{\|x\|}\bigg)x\Bigg).
\end{equation}
\item
\label{ex:ghiii}
Suppose that $\sigma>\gamma(\alpha-\rho^{q^*}/q^*)$ and 
$\|x\|\geq\gamma\rho^{q^*-1}(\sigma/\gamma+\rho^{2-q^*}+
\rho^{q^*}/q^*-\alpha)$. Then 
\begin{equation}
\label{e:gh4}
\prox_{\gamma\widetilde{\varphi}}(\sigma,x)=
\Bigg(\sigma+\gamma\bigg(\dfrac{\rho^{q^*}}{q^*}-\alpha\bigg),
\bigg(1-\dfrac{\gamma\rho}{\|x\|}\bigg)x\Bigg).
\end{equation}
\item
\label{ex:ghiv}
Suppose that
$\|x\|^{q^*}>q^*\gamma^{q^*}(\alpha-\sigma/\gamma)$ and 
$\|x\|<\gamma\rho^{q^*-1}(\sigma/\gamma+\rho^{2-q^*}+
\rho^{q^*}/q^*-\alpha)$. If $x\neq 0$, let $t$ be the unique
solution in $\RPP$ to the equation
\begin{equation}
\label{e:hb11}
\gamma t^{2q^*-1}+q^*(\sigma-\gamma\alpha)t^{q^*-1}
+\gamma q^*t-q^*\|x\|=0.
\end{equation}
Set $p=tx/\|x\|$ if $x\neq 0$, and $p=0$ if $x=0$.
Then 
\begin{equation}
\label{e:gh12}
\prox_{\gamma\widetilde{\varphi}}(\sigma,x)=
\begin{cases}
\big(\sigma+\gamma(t^{q^*}/q^*-\alpha),x-\gamma p\big),
&\text{if}\;\;q^*\gamma^{q^*-1}\sigma+\|x\|^{q^*}
>q^*\gamma^{q^*}\alpha;\\
\big(0,0\big),&\text{if}\;\;q^*\gamma^{q^*-1}\sigma+
\|x\|^{q^*}\leq q^*\gamma^{q^*}\alpha.
\end{cases}
\end{equation}
\end{enumerate}
\end{example}
\begin{proof}
We derive \eqref{e:hb1} from \eqref{e:gh},
\eqref{e:perspective}, and the fact that
$\rec\varphi=\rec(\rho\|\cdot\|)=\rho\|\cdot\|$. 
Now set 
\begin{equation}
\label{e:gh1}
\phi\colon\RR\to\RR\colon\xi\mapsto
\begin{cases}
\alpha-\Frac{\rho^{q^*}}{q^*}+\rho|\xi|,
&\text{if}\;\;|\xi|>\rho^{q^*/q};\\[4mm]
\alpha+\Frac{|\xi|^q}{q},&\text{if}\;\;|\xi|\leq\rho^{q^*/q}.
\end{cases}
\end{equation}
Then $\phi=(\rho|\cdot|)\infconv(|\cdot|^q/q)+\alpha$ is convex 
and even, and $\varphi=\phi\circ\|\cdot\|$. We derive from 
\cite[Proposition~13.24(i) and Example~13.2(i)]{Livre1} that
\begin{equation}
\label{e:gh5}
\phi^*=\Bigg(\big(\rho|\cdot|\big)\infconv
\bigg(\dfrac{|\cdot|^q}{q}\bigg)\Bigg)^*-\alpha
=\iota_{[-\rho,\rho]}+\dfrac{|\cdot|^{q^*}}{q^*}-\alpha.
\end{equation}
In turn, \eqref{e:gh5} and \eqref{e:R} yield
\begin{equation}
\label{e:R1}
\EuScript{R}=\EuScript{R}_1\cap\EuScript{R}_2,
\quad\text{where}\quad
\EuScript{R}_1=\RR\times[-\rho,\rho]\quad\text{and}\quad
\EuScript{R}_2=\menge{(\chi,\nu)\in\RR^2}
{|\nu|^{q^*}\leq q^*(\alpha-\chi)}.
\end{equation}
Now set 
$(\chi,\nu)=\proj_{\EuScript{R}}(\sigma/\gamma,\|x\|/\gamma)$.

\ref{ex:ghi}: This follows from
Proposition~\ref{p:1}\ref{p:1i} and \eqref{e:gh5}.

\ref{ex:ghii}: 
Since $\sigma/\gamma\leq\alpha-\rho^{q^*}/q^*$, we have
$(\chi,\nu)=\proj_{\EuScript{R}_1}(\sigma/\gamma,\|x\|/\gamma)
=(\sigma/\gamma,\proj_{[-\rho,\rho]}(\|x\|/\gamma))$.
Thus, since $\|x\|/\gamma>\rho$, $(\chi,\nu)=(\sigma/\gamma,\rho)$
and \eqref{e:gh3} follows from Proposition~\ref{p:1}\ref{p:1iii}.

\ref{ex:ghiii}:
The point $\Pi=(\alpha-\rho^{q^*}/q^*,\rho)$ is in the 
intersection of the boundaries of $\EuScript{R}_1$ and 
$\EuScript{R}_2$.
Therefore, the normal cone to $\EuScript{R}$ at $\Pi$ is
generated by outer normals $n_1$ to $\EuScript{R}_1$ and $n_2$ to 
$\EuScript{R}_2$ at $\Pi$. A tangent vector to $\EuScript{R}_2$ 
at $\Pi$ is 
$t(\Pi)=(-(|\cdot|^{q^*}/q^*)'(\rho),1)=(-\rho^{q^*-1},1)$.
We can take $n_1=(0,1)$ and $n_2=(1,\rho^{q^*-1})\perp t(\Pi)$.
Thus, the set of points which have projection $\Pi$ onto
$\EuScript{R}$ is 
\begin{align}
\Pi+N_{\EuScript{R}}\Pi
&=\Pi+\cone(n_1,n_2)\nonumber\\
&=\big(\alpha-\rho^{q^*}/q^*,\rho\big)+
\menge{(\tau,\xi)\in\RR\times\RR}
{\tau\geq 0\;\;\text{and}\;\xi\geq\rho^{q^*-1}\tau}\nonumber\\
&=\menge{(\tau,\xi)\in\RR\times\RR}
{\tau\geq\alpha-\rho^{q^*}/q^*\;\;\text{and}\;
\xi-\rho\geq\rho^{q^*-1}(\tau-\alpha+\rho^{q^*}/q^*)}\nonumber\\
&=\menge{(\tau,\xi)\in\RR\times\RR}
{\tau\geq\alpha-\rho^{q^*}/q^*\;\;\text{and}\;
\xi\geq\rho+\rho^{q^*-1}(\tau-\alpha)+\rho^{2q^*-1}/q^*)},
\end{align}
and therefore
\begin{equation}
(\chi,\nu)=\big(\alpha-\rho^{q^*}/q^*,\rho\big)
\quad\Leftrightarrow\quad
\begin{cases}
\sigma\geq\gamma\big(\alpha-\rho^{q^*}/q^*\big)\\
\|x\|\geq\gamma\big(\rho+\rho^{q^*-1}(\sigma/\gamma-\alpha)+
\rho^{2q^*-1}/q^*)\big).
\end{cases}
\end{equation}
In view of Proposition~\ref{p:1}\ref{p:1iii}, this yields
\eqref{e:gh4}.

\ref{ex:ghiv}: Here 
$(\sigma/\gamma,\|x\|/\gamma)\notin\EuScript{R}_2$ and 
$(\chi,\nu)=\proj_{\EuScript{R}_2}(\sigma/\gamma,\|x\|/\gamma)$.
Since 
\begin{equation}
\EuScript{R}_2=\menge{(\chi,\nu)\in\RR^2}
{\chi+\big(\alpha+|\cdot|^{q}/q\big)^*(\nu)\leq 0},
\end{equation}
the expression of $\prox_{\gamma\widetilde{\varphi}}(\sigma,x)$ is
computed exactly as though we were dealing with
the generalized scaled lasso function $\alpha+\|\cdot\|^{q}/q$ 
of Example~\ref{ex:sl} with $\kappa=q$ and the result is given in
\eqref{e:nyc2016-06-10b}.
\end{proof}

\begin{example}[generalized Berhu function]
\label{ex:bh}
Let $\alpha$, $\gamma$, $\rho$, and $\kappa$ be in $\RPP$, let 
$q\in\left]1,\pinf\right[$, and set $C=B(0;\rho)$. Define
$\varphi\colon\HH\to\RR$ by
\begin{equation}
\label{e:bh0}
\varphi=\alpha+\kappa\|\cdot\|+\frac{d^q_C}{q\rho^{q^*-1}},
\quad\text{where}\quad q^*=\dfrac{q}{q-1},
\end{equation}
and let $\sigma\in\RR$ and $x\in\HH$. Then 
\begin{equation}
\label{e:bh1}
\widetilde{\varphi}(\sigma,x)=
\begin{cases}
\alpha\sigma+\kappa\|x\|+\dfrac{\sigma}{q\rho^{q^*-1}}\bigg(
\dfrac{\|x\|}{\sigma}-\rho\bigg)^q,
&\text{if}\;\;\sigma>0\;\;\text{and}\;\;
\|x\|>\rho\sigma;\\
\alpha\sigma+\kappa\|x\|,
&\text{if}\;\;\sigma>0\;\;\text{and}\;\;
\|x\|\leq\rho\sigma;\\
0,&\text{if}\;\;\sigma=0\;\;\text{and}\;\;x=0;\\
\pinf,&\text{otherwise.}
\end{cases}
\end{equation}
Furthermore, set $\Delta\colon\RR\to\RR\colon\mu\mapsto
\text{max}(|\mu|-\kappa,0)+\text{max}^{q^*}
(|\mu|-\kappa,0)/q^*$.
Then the following hold:
\begin{enumerate}
\item
\label{ex:bhi}
Suppose that 
$\Delta(\|x\|/\gamma)\leq(\alpha-\sigma/\gamma)/\rho$.
Then $\prox_{\gamma\widetilde{\varphi}}(\sigma,x)=(0,0)$.
\item
\label{ex:bhiv}
Suppose that $\Delta(\|x\|/\gamma)>(\alpha-\sigma/\gamma)/\rho$
and that $\|x\|>\gamma\kappa+\rho(\sigma-\gamma\alpha)$.
If $x\neq 0$, let $t$ be the unique solution in
$\left]\kappa,\pinf\right[$ to the polynomial equation
\begin{equation}
\label{e:bh22}
\rho\bigg(\sigma-\gamma\alpha+\gamma\rho\bigg(t-\kappa+
\dfrac{(t-\kappa)^{q^*}}{q^*}\bigg)
\bigg)\Big(1+(t-\kappa)^{q^*-1}\Big)+\gamma t-\|x\|=0.
\end{equation}
Set $p=tx/\|x\|$ if $x\neq 0$, and set $t=0$ and $p=0$ if $x=0$.
Then $\prox_{\gamma\widetilde{\varphi}}(\sigma,x)=
\big(\sigma-\gamma\alpha+\gamma\rho\Delta(t),x-\gamma p\big)$.
\item
\label{ex:bhiii}
Suppose that 
$\gamma\kappa\leq\|x\|\leq\gamma\kappa+\rho(\sigma-\gamma\alpha)$.
Then 
\begin{equation}
\prox_{\gamma\widetilde{\varphi}}(\sigma,x)=
\big(\sigma-\gamma\alpha,(1-\gamma\kappa/\|x\|)x\big).
\end{equation}
\item
\label{ex:bhii}
Suppose that $\sigma>\gamma\alpha$ and $\|x\|<\gamma\kappa$. 
Then $\prox_{\gamma\widetilde{\varphi}}(\sigma,x)
=(\sigma-\gamma\alpha,0)$.
\end{enumerate}
\end{example}
\begin{proof}
The geometry underlying the proof is that depicted in 
Fig.~\ref{fig:1}, where $q=2$.
Set $R=[-\rho,\rho]$, $D=[-\kappa,\kappa]$,
$\phi=\alpha+\kappa|\cdot|+d^q_{[-\rho,\rho]}/(q\rho^{q/q^*})$,
$\theta\colon\RR\to\RR\colon t\mapsto|t|^{q}/(q\rho^{q/q^*})$, and
$\psi\colon\RR\to\RR\colon t\mapsto\rho(|t|+|t|^{q^*}/{q^*})$.
Then $\phi\colon\RR\to\RR$ is convex and even, and it follows
from \eqref{e:bh0} and \cite[Example~13.8]{Livre1} that 
\begin{equation}
\label{e:3g8h}
\varphi=\phi\circ\|\cdot\|\quad\text{and}\quad 
\varphi^*=\phi^*\circ\|\cdot\|.
\end{equation}
Furthermore,
$\sigma_{D}=\kappa|\cdot|$ and we derive from 
\cite[Examples~13.26 and 13.2(i)]{Livre1} that
\begin{align}
\label{e:ncw1}
\phi^*
&=\big(\sigma_D+\theta\circ d_R\big)^*-\alpha\nonumber\\
&=\sigma^*_D\infconv\big(\theta\circ d_R\big)^*-\alpha\nonumber\\
&=\iota_D\infconv\big(\sigma_R+\theta^*\circ|\cdot|\big)-\alpha
\nonumber\\
&=\iota_D\infconv\big(\rho|\cdot|+
\theta^*\circ|\cdot|\big)-\alpha\nonumber\\
&=\iota_D\infconv\big(\psi\circ|\cdot|\big)-\alpha\nonumber\\
&=\big(\psi\circ d_D\big)-\alpha\nonumber\\
&=\rho\bigg(d_D+\frac{d_D^{q^*}}{q^*}\bigg)-\alpha.
\end{align}
In turn, \cite[Example~17.33]{Livre1} yields
\begin{equation}
\label{e:bh7}
(\forall\nu\in\RR)\quad\partial\phi^*(\nu)=
\begin{cases}
\left\{\rho\bigg(\dfrac{1+d_D^{q^*-1}(\nu)}{d_D(\nu)}\bigg)
(\nu-\proj_D\nu)\right\},
&\text{if}\;\;\nu\notin D;\\[3mm]
N_D\nu\cap [-\rho,\rho],&\text{if}\;\;\nu\in D.
\end{cases}
\end{equation}
However, since $D=[-\kappa,\kappa]$, we have
$d_D\colon\nu\mapsto\max(|\nu|-\kappa,0)$. 
Therefore, \eqref{e:ncw1} implies that
\begin{equation}
\label{e:bh55}
(\forall\nu\in\RR)\quad\phi^*(\nu)=\rho\Delta(\nu)-\alpha=
\begin{cases}
\rho\big(|\nu|-\kappa+(|\nu|-\kappa)^{q^*}/q^*\big)-\alpha,
&\text{if}\;\;|\nu|>\kappa;\\
-\alpha,&\text{if}\;\;|\nu|\leq\kappa
\end{cases}
\end{equation}
and
\begin{equation}
\label{e:bh9}
(\forall\nu\in\RR)\quad\partial\phi^*(\nu)=
\begin{cases}
\left\{\rho\Big(1+(|\nu|-\kappa)^{q^*-1}\Big)
\sign(\nu)\right\},&\text{if}\;\;|\nu|>\kappa;\\[3mm]
[0,\rho],&\text{if}\;\;\nu=\kappa;\\
[-\rho,0],&\text{if}\;\;\nu=-\kappa;\\
\{0\},&\text{if}\;\;|\nu|<\kappa.
\end{cases}
\end{equation}
On the other hand, \eqref{e:bh55} and \eqref{e:R} yield
\begin{equation}
\label{e:R6}
\EuScript{R}=
\menge{(\chi,\nu)\in\RR^2}{\rho\Delta(\nu)\leq\alpha-\chi}.
\end{equation}
Now set $\Pi=(\alpha,\kappa)$ and
$(\chi,\nu)=\proj_{\EuScript{R}}(\sigma/\gamma,\|x\|/\gamma)$.
In view of \eqref{e:bh9}, the normal cone to $\epi\phi^*$ at
$(\kappa,-\alpha)$ is generated by the vectors $(\rho,-1)$ and 
$(0,-1)$. Hence, the normal cone to $\EuScript{R}$ at $\Pi$ is
generated by $n_1=(1,\rho)$ and $n_2=(1,0)$, that is 
\begin{equation}
N_{\EuScript{R}}\Pi=\menge{(\tau,\xi)\in\RR\times\RR}
{0\leq\xi\leq\rho\tau}.
\end{equation}
In turn, 
\begin{equation}
\label{e:kj85}
\proj_{\EuScript{R}}^{-1}\{\Pi\}=\Pi+N_{\EuScript{R}}\Pi=
\menge{(\tau,\xi)\in\RR^2}
{\kappa\leq\xi\leq\kappa+\rho(\tau-\alpha)}.
\end{equation}

\ref{ex:bhi}: It follows from the assumptions and \eqref{e:bh55} 
that $\sigma+\gamma\phi^*(\|x\|/\gamma)\leq 0$. In turn,
Proposition~\ref{p:1}\ref{p:1i} implies that
$\prox_{\gamma\widetilde{\varphi}}(\sigma,x)=(0,0)$.

\ref{ex:bhiv}: We have 
$(\sigma/\gamma,\|x\|/\gamma)\notin\EuScript{R}\supset
\left]\minf,\alpha\right]\times[-\kappa,\kappa]$ and 
$\|x\|/\gamma>\kappa+\rho(\sigma/\gamma-\alpha)$. Hence 
$|\nu|>\kappa$. Now set
$(\pi,p)=\proj_{\EuScript{C}}(\sigma/\gamma,x/\gamma)$. Then 
$\|p\|=|\nu|>\kappa$. Therefore, since it results from
\eqref{e:3g8h} and \eqref{e:ncw1} that 
$\dom\varphi^*=\dom(\phi^*\circ\|\cdot\|)=\HH$, 
Lemma~\ref{l:1}\ref{l:1ii},
\eqref{e:3g8h}, and \eqref{e:bh9} yield
\begin{align}
\label{e:bh20}
x
&=\gamma p+\big(\sigma+\gamma\varphi^*(p)\big)\nabla\varphi^*(p)
\nonumber\\
&=\gamma p+\big(\sigma+\gamma\phi^*(\|p\|)\big)
\nabla(\phi^*\circ\|\cdot\|)(p)
\nonumber\\
&=\Bigg(\gamma+\rho\bigg(\sigma-\gamma\alpha
+\gamma\rho\bigg(\|p\|-\kappa+
\dfrac{(\|p\|-\kappa)^{q^*}}{q^*}\bigg)
\bigg)\bigg(\dfrac{1+(\|p\|-\kappa)^{q^*-1}}
{\|p\|}\bigg)\Bigg)p.
\end{align}
Hence, 
\begin{equation}
\label{e:bh21}
p=\dfrac{1}{\gamma+\rho\bigg(\sigma-\gamma\alpha
+\gamma\rho\bigg(t-\kappa+
\dfrac{(t-\kappa)^{q^*}}{q^*}\bigg)\bigg)
\bigg(\dfrac{1+(t-\kappa)^{q^*-1}}{t}\bigg)}x,
\end{equation}
where $t=\|p\|$ is the unique solution in
$\left]\kappa,\pinf\right[$ to \eqref{e:bh22}, 
which is obtained by taking the norm of both sides of 
\eqref{e:bh20}. We then get the conclusion by invoking 
\eqref{e:ramm1}.

\ref{ex:bhiii}: In view of \eqref{e:kj85}, the assumptions imply 
that $(\sigma/\gamma,\|x\|/\gamma)\in\Pi+N_{\EuScript{R}}\Pi$
and therefore that $(\chi,\nu)=(\alpha,\kappa)$. Consequently,
Proposition~\ref{p:1}\ref{p:1iii} yields
$\prox_{\gamma\widetilde{\varphi}}(\sigma,x)=
\big(\sigma-\gamma\alpha,(1-\gamma\kappa/\|x\|)x\big)$.

\ref{ex:bhii}:
Set $H=\left]\minf,\alpha\right]\times\RR$. Then
$\EuScript{R}=\EuScript{R}\cap H$ and 
$(\sigma/\gamma,\|x\|/\gamma)\in
\left]\alpha,\pinf\right[\times[-\kappa,\kappa]$.
Hence, $(\chi,\nu)=\proj_H(\sigma/\gamma,\|x\|/\gamma)
=(\alpha,\|x\|/\gamma)$. In turn, 
we derive from Proposition~\ref{p:1}\ref{p:1iii} that
$\prox_{\gamma\widetilde{\varphi}}(\sigma,x)
=(\sigma-\gamma\alpha,0)$.
\end{proof}

\begin{example}[standard Berhu function]
\label{ex:berhu}
Let $\alpha$, $\gamma$, and $\rho$ be in $\RPP$. The standard
Berhu function of \cite{Owen07} with shift $\alpha$ is obtained
by setting $\HH=\RR$, $\kappa=1$, and $q=2$ in \eqref{e:bh0},
that is
\begin{equation}
\label{e:owen}
\mathsf{b}_\rho\colon\RR\to\RR\colon x\mapsto
\begin{cases}
\alpha+\dfrac{|x|^2+\rho^2}{2\rho},&\text{if}\;\:|x|>\rho;\\
\alpha+|x|,&\text{if}\;\:|x|\leq\rho.
\end{cases}
\end{equation}
Now let $\sigma\in\RR$ and $x\in\RR$. Then we derive 
from Example~\ref{ex:bh} that
\begin{equation}
\label{e:berhu}
\widetilde{\mathsf{b}}_\rho(\sigma,x)=
\begin{cases}
\alpha\sigma+\dfrac{|x|^2+\sigma^2\rho^2}{2\rho\sigma},
&\text{if}\;\;\sigma>0\;\;\text{and}\;\;
|x|>\sigma\rho;\\
\alpha\sigma+|x|,
&\text{if}\;\;\sigma>0\;\;\text{and}\;\;
|x|\leq\sigma\rho;\\
0,&\text{if}\;\;\sigma=0\;\;\text{and}\;\;x=0;\\
\pinf,&\text{otherwise,}
\end{cases}
\end{equation}
and that $\prox_{\gamma\widetilde{\mathsf{b}}_\rho}(\sigma,x)$ 
is given by
\begin{equation}
\begin{cases}
(0,0)&\text{if}\;\;
\text{max}(|x|^2-\gamma^2,0)\leq
2\gamma(\gamma\alpha-\sigma)/\rho;\\
(\sigma-\gamma\alpha,0)&\text{if}\;\;\sigma>\gamma\alpha
\;\;\text{and}\;\;|x|\leq\gamma;\\
\big(\sigma-\gamma\alpha,(1-\gamma/|x|)x\big)
&\text{if}\;\;\sigma>\gamma\alpha\;\;\text{and}\;\;
\gamma<|x|\leq\gamma+\rho(\sigma-\gamma\alpha);\\
\big(\sigma-\gamma\alpha+\gamma\rho(|p|^2-1)/2,x-\gamma p\big)
&\text{if}\;\;|x|>\gamma+\rho(\sigma-\gamma\alpha)
\;\;\text{and}\\
&\hskip 33mm |x|>\sqrt{\gamma^2+
2\gamma(\gamma\alpha-\sigma)/\rho},
\end{cases}
\end{equation}
with 
\begin{equation}
\label{e:nyc2017-07-23b}
p=\dfrac{1}{\gamma+
\rho\Big(\sigma-\gamma\alpha+\dfrac{\gamma\rho}{2}
\big(t^2-1\big)\Big)}\,x,
\end{equation}
where $t$ is the unique solution in $\left]1,\pinf\right[$ 
to the reduced third degree equation
\begin{equation}
\label{e:nyc2017-07-23c}
t^3+\bigg(\dfrac{2\big(\gamma+\rho(\sigma-\gamma\alpha)\big)}
{\gamma\rho^2}-1\bigg)t-\dfrac{2\|x\|}{\gamma\rho^2}=0,
\end{equation}
which can be solved explicitly via Cardano's formula.
\end{example}

\begin{example}[abstract Vapnik function]
\label{ex:v}
Let $\alpha$, $\varepsilon$, and $\gamma$ be in $\RPP$, and
define $\varphi\colon\HH\to\RR$ by
$\varphi=\alpha+\text{max}(\|\cdot\|-\varepsilon,0)$. Then 
\begin{equation}
\label{e:v1}
\widetilde{\varphi}\colon\HH\to\RX\colon(\sigma,x)\mapsto
\begin{cases}
\alpha\sigma+\text{max}(\|x\|-\varepsilon\sigma,0),
&\text{if}\;\;\sigma\geq 0;\\
\pinf,&\text{if}\;\;\sigma<0.
\end{cases}
\end{equation}
Now let $\sigma\in\RR$ and $x\in\HH$. Then the following hold:
\begin{enumerate}
\item
\label{ex:vi}
Suppose that $\sigma+\varepsilon\|x\|\leq\gamma\alpha$ and 
$\|x\|\leq\gamma$. Then 
$\prox_{\gamma\widetilde{\varphi}}(\sigma,x)=(0,0)$.
\item
\label{ex:vii}
Suppose that $\sigma\leq\gamma(\alpha-\varepsilon)$ and 
$\|x\|>\gamma$. Then 
\begin{equation}
\label{e:mg1}
\prox_{\gamma\widetilde{\varphi}}(\sigma,x)=
\Bigg(0,\bigg(1-\frac{\gamma}{\|x\|}\bigg)x\Bigg).
\end{equation}
\item
\label{ex:viii}
Suppose that $\sigma>\gamma(\alpha-\varepsilon)$ and $\|x\|\geq
\varepsilon\sigma+\gamma(1+\varepsilon(\varepsilon-\alpha))$. 
Then
\begin{equation}
\label{e:jk18}
\prox_{\gamma\widetilde{\varphi}}(\sigma,x)=
\Bigg(\sigma+\gamma(\varepsilon-\alpha),
\bigg(1-\frac{\gamma}{\|x\|}\bigg)x\Bigg).
\end{equation}
\item
\label{ex:viv}
Suppose that $\sigma+\varepsilon\|x\|>\gamma\alpha$ and 
$\varepsilon(\sigma-\gamma\alpha)<\|x\|<\varepsilon\sigma+
\gamma(1+\varepsilon(\varepsilon-\alpha))$. Then
\begin{equation}
\label{e:jk19}
\prox_{\gamma\widetilde{\varphi}}(\sigma,x)=
\frac{\sigma+\varepsilon\|x\|-\gamma\alpha}{1+\varepsilon^2}
\bigg(1,\dfrac{\varepsilon}{\|x\|}\,x\bigg).
\end{equation}
\item
\label{ex:vv}
Suppose that $\sigma\geq\gamma\alpha$ and
$\|x\|\leq\varepsilon(\sigma-\gamma\alpha)$. Then 
$\prox_{\gamma\widetilde{\varphi}}(\sigma,x)=
(\sigma-\gamma\alpha,x)$.
\end{enumerate}
\end{example}
\begin{proof}
We derive \eqref{e:v1} at once from \eqref{e:perspective}.
Set $\phi=\alpha+\text{max}(|\cdot|-\varepsilon,0)$. Then 
$\varphi=\phi\circ\|\cdot\|$ and 
$\phi=\alpha+d_{[-\varepsilon,\varepsilon]}=
\alpha+\iota_{[-\varepsilon,\varepsilon]}\infconv|\cdot|$. 
Therefore
\begin{equation}
\label{e:dh1}
\phi^*=\varepsilon|\cdot|+\iota_{[-1,1]}-\alpha.
\end{equation}
Thus, \eqref{e:R} yields
\begin{equation}
\label{e:dh2}
\EuScript{R}=\EuScript{R}_1\cap\EuScript{R}_2,
\quad\text{where}\quad
\EuScript{R}_1=\left]\minf,\alpha\right]\times[-1,1]
\quad\text{and}\quad
\EuScript{R}_2=\menge{(\chi,\nu)\in\RR^2}
{\varepsilon|\nu|\leq\alpha-\chi}.
\end{equation}
Now set 
$(\chi,\nu)=\proj_{\EuScript{R}}(\sigma/\gamma,\|x\|/\gamma)$.

\ref{ex:vi}: This follows from \eqref{e:dh1} and
Proposition~\ref{p:1}\ref{p:1i}.

\ref{ex:vii}: Since $\sigma/\gamma\leq\alpha-\varepsilon$ and
$\|x/\gamma\|>1$, it follows from \eqref{e:dh2} that
\begin{equation}
(\chi,\nu)
=\proj_{\EuScript{R}_1}\big(\sigma/\gamma,\|x\|/\gamma\big)
=\big(\sigma/\gamma,1\big).
\end{equation}
In turn, we derive \eqref{e:mg1} from 
Proposition~\ref{p:1}\ref{p:1iii}.

\ref{ex:viii}:
The point $\Pi=(\alpha-\varepsilon,1)$ lies in the intersection
of the boundaries of $\EuScript{R}_1$ and $\EuScript{R}_2$, which
are line segments.
Therefore, the normal cone to $\EuScript{R}$ at $\Pi$ is
generated by outer normals $n_1$ to $\EuScript{R}_1$ and $n_2$ to 
$\EuScript{R}_2$ at $\Pi$. A tangent vector to $\EuScript{R}_2$ at
$\Pi$ is $t(\Pi)=(-\varepsilon,1)$. Therefore we
take $n_1=(0,1)$ and $n_2=(1,\varepsilon)\perp t(\Pi)$.
Consequently, the set of points which have projection $\Pi$ onto
$\EuScript{R}$ is 
\begin{align}
\proj_{\EuScript{R}}^{-1}\{\Pi\}
&=\Pi+N_{\EuScript{R}}\Pi\nonumber\\
&=\Pi+\cone(n_1,n_2)\nonumber\\
&=\big(\alpha-\varepsilon,1\big)+
\menge{(\tau,\xi)\in\RR^2}
{\tau\geq 0\;\;\text{and}\;\xi\geq\varepsilon\tau}\nonumber\\
&=\menge{(\tau,\xi)\in\RR^2}
{\tau\geq\alpha-\varepsilon\;\;\text{and}\;
\xi\geq 1+\varepsilon(\tau-\alpha+\varepsilon)}\nonumber\\
&=\menge{(\tau,\xi)\in\RR^2}
{\tau\geq\alpha-\varepsilon\;\;\text{and}\;\xi\geq
\varepsilon\tau+1+\varepsilon(\varepsilon-\alpha)},
\end{align}
and it contains $(\sigma/\gamma,\|x\|/\gamma)$.
Hence
\begin{equation}
(\chi,\nu)=\big(\alpha-\varepsilon,1\big)
\quad\Leftrightarrow\quad
\begin{cases}
\sigma\geq\gamma(\alpha-\varepsilon)\\
\|x\|\geq\varepsilon\sigma+\gamma
\big(1+\varepsilon(\varepsilon-\alpha)\big).
\end{cases}
\end{equation}
We then use Proposition~\ref{p:1}\ref{p:1iii} to get 
\eqref{e:jk18}.

\ref{ex:viv}:
In this case, $(\chi,\nu)=\proj_{\EuScript{R}_2}
(\sigma/\gamma,\|x\|/\gamma)$. More precisely, $(\chi,\nu)$ is 
the projection of $(\sigma/\gamma,\|x\|/\gamma)$ onto the hyperplane
$\menge{(\tau,\xi)\in\RR^2}{\varepsilon\xi\leq\alpha-\tau}
=\menge{(\tau,\xi)\in\RR^2}{\scal{(\tau,\xi)}{n_2}\leq\alpha}$, 
where $n_2=(1,\varepsilon)$. Thus,
\begin{align}
(\chi,\nu)
&=\dfrac{1}{\gamma}(\sigma,\|x\|)+
\dfrac{\alpha-\scal{(\sigma,\|x\|)}{n_2}/\gamma}
{\|n_2\|^2}n_2\nonumber\\
&=\bigg(\dfrac{\sigma}{\gamma}+
\dfrac{\alpha-(\sigma+\varepsilon\|x\|)/\gamma}{1+\varepsilon^2},
\dfrac{\|x\|}{\gamma}+\varepsilon
\dfrac{\alpha-(\sigma+\varepsilon\|x\|)/\gamma}{1+\varepsilon^2}
\bigg),
\end{align}
and \eqref{e:jk19} follows from 
Proposition~\ref{p:1}\ref{p:1iii}.

\ref{ex:vv}:
Set $\Pi=(\alpha,0)$, $n_2=(1,\varepsilon)$, and 
$n_3=(1,-\varepsilon)$. The set of points which have projection
$\Pi$ onto
$\EuScript{R}$ is 
\begin{align}
\Pi+N_{\EuScript{R}}\Pi
&=\Pi+N_{\EuScript{R}_2}\Pi\nonumber\\
&=\Pi+\cone(n_2,n_3)\nonumber\\
&=\big(\alpha,0\big)+\menge{(\tau,\xi)\in\RR^2}
{\tau\geq 0\;\;\text{and}\;\xi\leq\varepsilon\tau}\nonumber\\
&=\menge{(\tau,\xi)\in\RR^2}{\tau\geq\alpha\;\;\text{and}\;
\xi\leq\varepsilon(\tau-\alpha)},
\end{align}
and it therefore contains $(\sigma/\gamma,\|x\|/\gamma)$.
In turn, $(\chi,\nu)=(\alpha,0)$ and the conclusion follows from 
Proposition~\ref{p:1}\ref{p:1iii}.
\end{proof}

\section{Optimization model and examples}
\label{sec:3}
Let us first recall that our data formation model is 
Model~\ref{m:1}. We now introduce the perspective M-estimation
model, which enables the estimation of the regression vector
$\overline{b}=(\overline{\beta}_k)_{1\leq k\leq p}\in\RR^p$ as 
well as scale vectors
$\overline{s}=(\overline{\sigma}_i)_{1\leq i\leq N}\in\RR^N$ and 
$\overline{t}=(\overline{\tau}_i)_{1\leq i\leq P}\in\RR^P$. 
If robust data fitting functions are used, the outlier vector in
Model~\ref{m:1} can be identified from the solution of
\eqref{e:prob0} below. 
For instance, if the Huber function is used for data fitting, 
one can estimate the mean shift vector $\overline{o}$ in 
\eqref{e:1} \cite{Anto07,Shey11}. 

The optimization problem under investigation is as follows.

\begin{problem}
\label{prob:1}
Let $N$ and $P$ be strictly positive integers,
let $\varsigma\in\Gamma_0(\RR^N)$, let $\varpi\in\Gamma_0(\RR^P)$,
let $\theta\in\Gamma_0(\RR^p)$, let $(n_i)_{1\leq i\leq N}$
be strictly positive integers such that $\sum_{i=1}^Nn_i=n$, and
let $(p_i)_{1\leq i\leq P}$ be strictly positive integers.
For every $i\in\{1,\ldots,N\}$, let 
$\varphi_i\in\Gamma_0(\RR^{n_i})$,
let $X_i\in\RR^{n_i\times p}$,
and let $y_i\in\RR^{n_i}$ be such that 
\begin{equation}
\label{e:blocks}
X=
\begin{bmatrix}
X_1\\
\vdots\\
X_N\\
\end{bmatrix}
\quad\text{and}\quad
y=
\begin{bmatrix}
y_1\\
\vdots\\
y_N\\
\end{bmatrix}.
\end{equation}
Finally, for every $i\in\{1,\ldots,P\}$, let 
$\psi_i\in\Gamma_0(\RR^{p_i})$, and 
let $L_i\in\RR^{p_i\times p}$.
The objective of perspective M-estimation is to 
\begin{equation}
\label{e:prob0}
\minimize{s\in\RR^N,\,t\in\RR^P,\,b\in\RR^p}
\varsigma(s)+\varpi(t)+\theta(b)
+{\Sum_{i=1}^{N}\widetilde{\varphi}_i
\big(\sigma_i,X_ib-y_i\big)+\Sum_{i=1}^{P}
\widetilde{\psi}_i\big(\tau_i,L_ib\big)}.
\end{equation}
\end{problem}

\begin{remark}
\label{r:prob1}
Let us make a few observations about Problem~\ref{prob:1}.
\begin{enumerate}
\item
In \eqref{e:prob0}, $N+P$ perspective functions 
$(\widetilde{\varphi_i})_{1\leq i\leq N}$ and
$(\widetilde{\psi_i})_{1\leq i\leq P}$ are used to penalize affine
transformations $(X_ib-y_i)_{1\leq i\leq N}$ and
$(L_ib)_{1\leq i\leq P}$ of $b$. The operators 
$(L_i)_{1\leq i\leq P}$ can, for instance, select a single
coordinate, or blocks of coordinates (as in the group lasso
penalty), or can model finite difference operators.
Constraints on the scale 
variables $(\sigma_i)_{1\leq i\leq N}$ and $(\tau_i)_{1\leq i\leq P}$
of the perspective functions can be enforced via the functions
$\varsigma$ and $\varpi$. 
\item
It is also possible to use ``scaleless'' non-perspective functions
of the transformations $(X_ib-y_i)_{1\leq i\leq N}$ and
$(L_ib)_{1\leq i\leq P}$. For instance, given
$i\in\{1,\ldots,N\}$, the term ${\varphi}_i(X_ib-y_i)$ is obtained 
by using $\widetilde{\varphi}_i(\sigma_i,X_ib-y_i)$ and
imposing $\sigma_i=1$ via $\varsigma$.
\item
We attach individual scale variables to each of the functions
$(\widetilde{\varphi_i})_{1\leq i\leq N}$ and
$(\widetilde{\psi_i})_{1\leq i\leq P}$ for 
flexibility in the case of heteroscedastic models, but also 
for computational reasons. Indeed, the
proximal tools we are proposing in Sections~\ref{sec:4} and
\ref{sec:5} can handle separable functions better. For instance,
it is hard to process the function 
\begin{equation}
(\sigma,x_1,x_2)\mapsto\widetilde{\varphi}_1(\sigma,x_1)
+\widetilde{\varphi}_2(\sigma,x_2)
\end{equation}
via proximal tools, whereas the equivalent separable function with 
coupling of the scales 
\begin{multline}
(\sigma_1,\sigma_2,x_1,x_2)\mapsto\varsigma(\sigma_1,\sigma_2)+
\widetilde{\varphi}_1(\sigma_1,x_1)
+\widetilde{\varphi}_2(\sigma_2,x_2),\\
\quad\text{where}\quad
\varsigma(\sigma_1,\sigma_2)=
\begin{cases}
0,&\text{if}\;\;\sigma_1=\sigma_2;\\
\pinf,&\text{if}\;\;\sigma_1\neq\sigma_2,
\end{cases}
\end{multline}
will be much easier.
\end{enumerate}
\end{remark}

We now present some important instantiations of
Problem~\ref{prob:1}.

\begin{example}
\label{ex:1}
Consider the optimization problem
\begin{equation}
\label{e:ridge}
\minimize{b\in\RR^p}{\|Xb-y\|_q^q+\alpha_1\|b\|_1
+\alpha_2\|b\|_r^r},
\end{equation}
where $\alpha_1\in\RP$, $\alpha_2\in\RP$, $q\in\{1,2\}$,
and $r\in[1,2]$.
For $q=r=2$, $\alpha_1>0$, and $\alpha_2>0$, \eqref{e:ridge} 
is the elastic-net model of \cite{Zouz05}; in addition, if 
$\alpha_1=0$ and $\alpha_2>0$, we obtain the 
ridge regression model \cite{Hoer62} and, if $\alpha_1>0$ and
$\alpha_2=0$, we obtain the lasso model \cite{Tibs96}. On the other
hand, taking $q=1$, $\alpha_1>0$, and $\alpha_2=0$, leads to
the least absolute deviation lasso model of \cite{Xuji10}. 
Finally, taking $q=2$, $\alpha_1=0$, and $\alpha_2>0$ yields to 
the bridge model \cite{Fran93}. The formulation \eqref{e:ridge} 
corresponds to the special case of Problem~\ref{prob:1} in which 
\begin{equation}
\label{e:kj60}
\begin{cases}
N=1,\;n_1=n,\;\varphi_1=\|\cdot\|_q^q\\
P=1,\;p_1=p,\;\psi_1=0,\;L_1=0\\
\varsigma=\iota_{\{1\}},\;
\varpi=0,\;
\theta=\alpha_1\|\cdot\|_1+\alpha_2\|\cdot\|_r^r.
\end{cases}
\end{equation}
Note that our choice of $\varsigma$ imposes that 
$\sigma_1=1$ and therefore that
$\widetilde{\varphi}_1(\sigma_1,\cdot)=\|\cdot\|_q^q$. 
The proximity operator of $\varphi_1$ is derived in 
\cite{Chau07} and that of $\theta$ in \cite{Siop07}.
\end{example}

\begin{example}
Given $\alpha_1$ and $\alpha_2$ in $\RP$ and $q\in\{1,2\}$, 
consider the model 
\begin{equation}
\label{e:fused}
\minimize{b\in\RR^p}{\|Xb-y\|_2^2+\alpha_1\sum_{i=1}^p|\beta_i|
+\alpha_2\sum_{i=1}^{p-1}|\beta_{i+1}-\beta_i|^q}.
\end{equation}
It derives from Problem~\ref{prob:1} by setting 
\begin{equation}
\label{e:kj61}
\begin{cases}
N=1,\;n_1=n,\;\varphi_1=\|\cdot\|_2^2\\
P=p-1,\;(\forall i\in\{1,\ldots,P\})\;\; p_i=1,\; 
\psi_i=\alpha_2|\cdot|^q,\;L_i\colon b\mapsto\beta_{i+1}-\beta_i\\
\varsigma=\iota_{\{1\}},\;
\varpi=\iota_{\{(1,\ldots,1)\}},\;
\theta=\alpha_1\|\cdot\|_1.
\end{cases}
\end{equation}
For $q=1$, we obtain the fused lasso model \cite{Tibs05}, while 
$q=2$ yields the smooth lasso formulation of \cite{Hebi11}.
Let us note that 
one obtains alternative formulations such that of
\cite{Tibs14} by suitably redefining the operators 
$(L_i)_{1\leq i\leq P}$ in \eqref{e:kj61}.
\end{example}

\begin{example}
\label{ex:3}
Given $\rho_1$ and $\rho_2$ in $\RPP$, the formulation proposed in
\cite{Owen07} is
\begin{equation}
\label{e:owen1}
\minimize{\sigma\in\RPP,\,\tau\in\RPP,\,b\in\RR^p}
{\sigma\sum_{i=1}^n
\mathsf{h}_{\rho_1}\bigg(\frac{x_i^\top b-\eta_i}{\sigma}\bigg)
+n\sigma+\alpha_1\tau\sum_{i=1}^p{\mathsf{b}_{\rho_2}
\bigg(\frac{\beta_i}{\tau}\bigg)}+p\tau,}
\end{equation}
where $\mathsf{h}_{\rho_1}$ and
$\mathsf{b}_{\rho_2}$ are the Huber and Berhu functions of 
\eqref{e:huber} and \eqref{e:berhu}, respectively. 
From a convex optimization viewpoint, we reformulate this problem
more formally in terms of the lower semicontinuous function
of \eqref{e:perspective} to obtain
\begin{equation}
\label{e:owen2}
\minimize{\sigma\in\RR,\,\tau\in\RR,\,b\in\RR^p}
{\sum_{i=1}^n[\mathsf{h}_{\rho_1}+n]^\sim
\big(\sigma,x_i^\top b-\eta_i\big)
+\alpha_1\sum_{i=1}^p[\mathsf{b}_{\rho_2}+p]^\sim
\big(\tau,\beta_i\big).}
\end{equation}
This is a special case of Problem~\ref{prob:1} with
\begin{equation}
\label{e:kj62}
\begin{cases}
N=n\;\;\text{and}\;\;(\forall i\in\{1,\ldots,N\})\;\; n_i=1,\; 
\varphi_i=\mathsf{h}_{\rho_1}+n,\;X_i=x_i^\top\\
P=p\;\;\text{and}\;\;(\forall i\in\{1,\ldots,P\})\;\; p_i=1,\; 
\psi_i=\alpha_1\mathsf{b}_{\rho_2}+p,\;L_i\colon b\mapsto
\beta_i\\
\varsigma=\iota_D,\;\;\text{where}\;\;
D=\menge{(\sigma,\ldots,\sigma)\in\RR^n}{\sigma\in\RR}\\
\varpi=\iota_E,\;\;\text{where}\;\;
E=\menge{(\tau,\ldots,\tau)\in\RR^p}{\tau\in\RR}\\
\theta=0.
\end{cases}
\end{equation}
If one omits the right-most summation in \eqref{e:owen2} one
recovers Huber's concomitant model \cite{Hube81}. Note that
\begin{equation}
\label{e:turk}
\prox_{\varsigma}=\proj_D\colon(\sigma_i)_{1\leq i\leq n}\mapsto
\Bigg(\frac{1}{n}\sum_{i=1}^n\sigma_i,\ldots,
\frac{1}{n}\sum_{i=1}^n\sigma_i\Bigg).
\end{equation}
The operator $\prox_{\varpi}$ is computed likewise. On the other
hand, the proximity operators of 
$\mathsf{h}_{\rho_1}$ and $\mathsf{b}_{\rho_2}$ are provided in
Examples~\ref{ex:gh} and \ref{ex:bh}, respectively.
\end{example}

\begin{example}
\label{ex:4}
The scaled square-root elastic net formulation of \cite{Rani17} is
\begin{equation}
\label{e:rani1}
\minimize{\sigma\in\RPP,\,b\in\RR^p}
{\frac{\|Xb-y\|_2^2}{2\sigma}+\frac{n\sigma}{2}
+\alpha_1\|b\|_1+\alpha_2\|b\|_2^q},
\end{equation}
where $\alpha_1\in\RP$, $\alpha_2\in\RP$, and $q\in\{1,2\}$.
Reformulated more formally in terms of lower semicontinuous 
functions, this model becomes
\begin{equation}
\label{e:rani2}
\minimize{\sigma\in\RR,\,b\in\RR^p}
{\bigg[\frac{\|\cdot\|_2^2+n}{2}\bigg]^\sim\big(\sigma,Xb-y\big)
+\alpha_1\|b\|_1+\alpha_2\|b\|_2^q}.
\end{equation}
We thus obtain the special case of Problem~\ref{prob:1} in which
\begin{equation}
\label{e:rani3}
\begin{cases}
N=1,\; n_1=n,\;\varphi_1=\big(\|\cdot\|_2^2+n\big)/2\\
P=p\;\;\text{and}\;\;(\forall i\in\{1,\ldots,P\})\;\; p_i=1,\; 
\psi_i=\alpha_1|\cdot|,\;L_i\colon b\mapsto\beta_i\\
\varsigma=0,\;\varpi=0,\;\theta=\alpha_2\|b\|_2^q.
\end{cases}
\end{equation}
The proximity operator of $\theta$ is given in \cite{Smms05}, 
while that of $\widetilde{\varphi}_1$ is provided in 
Example~\ref{ex:sl}.
Note that, when $q=2$, we could also take the functions
$(\psi_i)_{1\leq i\leq P}$ to be zero and 
$\theta=\alpha_1\|b\|_1+\alpha_2\|b\|_2^2$ since the proximity
operator of $\theta$ is computable explicitly in this case
\cite{Siop07}. When $\alpha_2=0$ in \eqref{e:rani2}, we obtain 
the scaled lasso model \cite{Anto10,Sunt12}. On the other hand if
we use $\alpha_2=0$ and
$\varsigma=\iota_{\left[\varepsilon,\pinf\right[}$ for some
$\varepsilon\in\RPP$ in \eqref{e:rani2},
we recover the formulation of \cite{Ndia17}.
\end{example}

\begin{example}
\label{ex:9}
Given $\alpha$, $\rho_1$, $\rho_2$, 
and $(\omega_i)_{1\leq i\leq p}$ in $\RPP$, the formulation 
proposed in \cite{Lamb16} is
\begin{multline}
\label{e:lamb1}
\minimize{\sigma\in\RR,\,\tau\in\RR,\,b\in\RR^p}
{
\begin{cases}
\sigma\Sum_{i=1}^n
\mathsf{h}_{\rho_1}\bigg(\dfrac{x_i^\top b-\eta_i}
{\sigma}\bigg)+n\sigma,&\text{if}\;\;\sigma>0;\\
\rho_1\Sum_{i=1}^n\big|x_i^\top b-\eta_i\big|,
&\text{if}\;\;\sigma=0;\\
\pinf,&\text{if}\;\;\sigma<0
\end{cases}\\[4mm]
+
\begin{cases}
\alpha\tau\Sum_{i=1}^p\bigg(\omega_i{\mathsf{b}_{\rho_2}\bigg(
\dfrac{\beta_i}{\tau}\bigg)}+\frac{1}{\omega_i}\bigg),
&\text{if}\;\;\tau>0;\\
0,&\text{if}\;\;b=0\;\text{and}\;\tau=0;\\
\pinf,&\text{otherwise,}
\end{cases}
}
\end{multline}
where $\mathsf{h}_{\rho_1}$ and $\mathsf{b}_{\rho_2}$ are the
Huber and Berhu functions of
\eqref{e:huber} and \eqref{e:berhu}, respectively.  In view of
\eqref{e:perspective}, we can rewrite \eqref{e:lamb1} as 
\begin{equation}
\label{e:lamb28}
\minimize{\sigma\in\RR,\,\tau\in\RR,\,b\in\RR^p}
{\sum_{i=1}^n[\mathsf{h}_{\rho_1}+n]^\sim
\big(\sigma,x_i^\top b-\eta_i\big)+\alpha\sum_{i=1}^p
\bigg[\omega_i\mathsf{b}_{\rho_2}+\dfrac{1}{\omega_i}\bigg]^\sim
\big(\tau,\beta_i\big).}
\end{equation}
This is a special case of Problem~\ref{prob:1} with
\begin{equation}
\label{e:lamb3}
\begin{cases}
N=n\;\;\text{and}\;\;(\forall i\in\{1,\ldots,N\})\;\; n_i=1,\; 
\varphi_i=\mathsf{h}_{\rho_1}+n,\;X_i=x_i^\top\\
P=p\;\;\text{and}\;\;(\forall i\in\{1,\ldots,P\})\;\; p_i=1,\; 
\psi_i=\alpha\big(\omega_i\mathsf{b}_{\rho_2}+1/\omega_i\big),\;
L_i\colon b\mapsto\beta_i\\
\varsigma=\iota_D,\;\;\text{where}\;\;
D=\menge{(\sigma,\ldots,\sigma)\in\RR^n}{\sigma\in\RR}\\
\varpi=\iota_E,\;\;\text{where}\;\;
E=\menge{(\tau,\ldots,\tau)\in\RR^p}{\tau\in\RR}\\
\theta=0.
\end{cases}
\end{equation}
The variant studied in \cite{Lamb11} replaces the functions
$(\psi_i)_{1\leq i\leq p}$ of \eqref{e:lamb3} by
$(\forall i\in\{1,\ldots,p\})$ $\psi_i=\alpha\omega_i|\beta_i|$.
\end{example}

\begin{example}
\label{ex:7}
Let $\alpha\in\RPP$. The formulation 
\begin{equation}
\label{e:bien1}
\minimize{\tau\in\RPP,\,b\in\RR^p}
{-\dfrac{\ln\tau}{2}+\dfrac{\|y\|_2^2\,\tau}{2n}
+\frac{\|Xb\|_2^2}{2n\tau}
+\alpha\|b\|_1-\dfrac{{y}^\top{Xb}}{n}},
\end{equation}
was proposed in \cite{Bien18} under the name ``natural lasso."
It can be cast in the framework of Problem~\ref{prob:1} with
\begin{equation}
\label{e:bien2}
\begin{cases}
N=1,\; n_1=n,\;\varphi_1=0\\
P=1,\; p_1=p,\;\psi_1=\|\cdot\|_2^2/(2n),\;L_1=X\\
\varsigma=0,\;
\theta=\alpha\|\cdot\|_1-\scal{X^\top y}{\cdot}/n
\end{cases}
\end{equation}
and
\begin{equation}
\varpi\colon\tau\mapsto
\begin{cases}
-(\ln\tau)/2+\|y\|_2^2\tau/(2n),&\text{if}\;\;\tau>0;\\
\pinf,&\text{if}\;\;\tau\leq 0.
\end{cases}
\end{equation}
The proximity operators of $\theta$ and $\varpi$ are given in 
\cite{Smms05}.
\end{example}

\begin{example}
\label{ex:8}
Given $\alpha$ and $\varepsilon$
in $\RPP$, define $\mathsf{v}_i\colon\RR\to\RR\colon\eta\mapsto
\alpha+\text{max}(|\eta|-\varepsilon,0)$. Using the perspective 
function derived in Example~\ref{ex:v}, we can rewrite the 
linear $\nu$-support vector regression problem of \cite{Scho00} as
\begin{equation}
\label{e:nusvm}
\minimize{\sigma\in\RR,\,b\in\RR^p}
{\sum_{i=1}^n\widetilde{\mathsf{v}}_i
\big(\sigma,x_i^\top b-\eta_i\big)+\frac{1}{2}\|b\|^2_2}.
\end{equation}
We identify this problem as a special case of 
Problem~\ref{prob:1} with
\begin{equation}
\label{e:nusvm2}
\begin{cases}
N=n \;\;\text{and}\;\;(\forall i\in\{1,\ldots,N\})\;\;
\varphi_i=\mathsf{v}_i,\;X_i=x_i^\top\\
P=1,\; p_1=p,\;\psi_1=0,\;L_1=0\\
\varsigma=\iota_D,\;\;\text{where}\;\;
D=\menge{(\sigma,\ldots,\sigma)\in\RR^n}
{\sigma\in\RR}\\
\varpi=0,\;\theta=\|\cdot\|^2_2/2.
\end{cases}
\end{equation}
The proximity operator of $\widetilde{\mathsf{v}}_i$ is given in
Example~\ref{ex:v} and that of $\varsigma$ in \eqref{e:turk}.
The concomitant parameter $\sigma$ scales the width of the
``tube'' in the $\nu$-support vector regression and trades off
model complexity and slack variables \cite{Scho00}. 
\end{example}

The next two examples are novel M-estimators that will be employed
in Section~\ref{sec:5}.

\begin{example}
\label{ex:24a}
In connection with \eqref{e:blocks}, we introduce a generalized
heteroscedastic scaled lasso with $N$ data blocks, which employs
the perspective derived in Example~\ref{ex:sl}. Recall that $n_i$
is the number of data points in the $i$th block, let
$\alpha_1\in\RP$, and set
\begin{equation}
\label{e:u83}
(\forall i\in\{1,\ldots,N\})\quad
\mathsf{c}_{i,q}\colon\RR^{n_i}\to\RR\colon x\mapsto
\|x\|_2^q+\dfrac{1}{2}.
\end{equation}
The objective is to 
\begin{equation}
\label{e:lamb2}
\minimize{s\in\RR^N,\,b\in\RR^p}
{\sum_{i=1}^N\widetilde{\mathsf{c}}_{i,q}
\big(\sigma_i,X_ib-y_i\big)+\alpha_1\|b\|_1}.
\end{equation}
This is a special case of Problem~\ref{prob:1} with
\begin{equation}
\label{e:lamb4}
\begin{cases}
(\forall i\in\{1,\ldots,N\})\;\;
\varphi_i=\mathsf{c}_{i,q}\\
P=1,\; p_1=p,\;\psi_1=0,\;L_1=0\\
\varpi=0,\;\varsigma=0,\;\theta=\alpha_1\|\cdot\|_1.
\end{cases}
\end{equation}
The choice of the exponent $q\in\left]1,\pinf\right[$ reflects
prior distributional assumptions on the noise. This model can
handle generalized normal distributions. The proximity operator of
$\widetilde{\mathsf{c}}_{i,q}$ is provided in Example~\ref{ex:sl}.
\end{example}

\begin{example}
\label{ex:24b}
In connection with \eqref{e:blocks}, 
we introduce a generalized heteroscedastic Huber M-estimator,
with $J$ scale variables $(\sigma_j)_{1\leq j\leq J}$, which 
employs the perspective derived in
Example~\ref{ex:gh}. Each scale $\sigma_j$ is attached to a 
group of $m_j$ data points, hence $\sum_{j=1}^Jm_j=n$.
Let $\alpha_1$ and $\alpha_2$ be in $\RP$, let $\delta$, $\rho_1$, 
and $\rho_2$ be in $\RPP$, and denote by $\mathsf{h}_{\rho_1,q}$
the function in \eqref{e:gh}, where $\HH=\RR$. The
objective is to 
\begin{equation}
\label{e:lamb5}
\minimize{s\in\RR^J,\,\tau\in\RR,\,b\in\RR^p}
{\sum_{j=1}^J\sum_{i=1}^{m_j}[\mathsf{h}_{\rho_1,q}+\delta]^\sim
\big(\sigma_j,x_i^\top b-\eta_i\big)+\alpha_1\|b\|_1
+\alpha_2\sum_{i=1}^p[\mathsf{b}_{\rho_2}+p]^\sim
\big(\tau,\beta_i\big).}
\end{equation}
This statistical model is rewritten in the format of the
computational model described in Problem~\ref{prob:1} by choosing
\begin{equation}
\label{e:kj82}
\begin{cases}
N=n\;\;\text{and}\;\;(\forall i\in\{1,\ldots,N\})\;\; n_i=1,\; 
\varphi_i=\mathsf{h}_{\rho_1,q}+\delta,\;X_i=x_i^\top\\
P=p\;\;\text{and}\;\;(\forall i\in\{1,\ldots,P\})\;\; p_i=1,\; 
\psi_i=\alpha_2\mathsf{b}_{\rho_2}+p,\;L_i\colon b\mapsto\beta_i\\
\varsigma=\iota_D,\;\;\text{where}\;\;
D=\menge{(\sigma_1,\ldots,\sigma_1,\ldots,\sigma_J,\ldots,\sigma_J)
\in\RR^n}{(\sigma_j)_{1\leq j\leq J}\in\RR^J}\\
\varpi=\iota_E,\;\;\text{where}\;\;
E=\menge{(\tau,\ldots,\tau)\in\RR^p}{\tau\in\RR}\\
\theta=\alpha_1\|\cdot\|_1.
\end{cases}
\end{equation}
The choice of the exponent $q\in\left]1,\pinf\right[$ reflects
prior distributional assumptions on the noise. This model
handles generalized normal distributions and can identify outliers. 
Note that 
\begin{equation}
\prox_{\varsigma}=\proj_D\colon(\sigma_i)_{1\leq i\leq n}\mapsto
\Bigg(\underbrace{\frac{1}{m_1}\sum_{i=1}^{m_1}\sigma_i,\ldots,
\frac{1}{m_1}\sum_{i=1}^{m_1}\sigma_i}_{m_1\;\text{terms}},\ldots,
\underbrace{\frac{1}{m_J}\sum_{i=n-m_{J}+1}^{n}\sigma_i,\ldots,
\frac{1}{m_J}\sum_{i=n-m_{J}+1}^{n}\sigma_i}_{m_J\;\text{terms}}
\Bigg).
\end{equation}
\end{example}

\begin{remark}
Particular instances of perspective M-estimation models come with
statistical guarantees. For the scaled lasso, initial theoretical
guarantees are given in \cite{Sunt12}. In \cite{Lamb11,Lamb16}
results are provided for the homoscedastic Huber
M-estimator with adaptive $\ell^1$ penalty and the adaptive Berhu
penalty. In \cite{Hann17}, explicit bounds for
estimation and prediction error for ``convex loss lasso'' problems
are given which cover scaled homoscedastic lasso, the least
absolute deviation model, and the homoscedastic Huber model. For
the heteroscedastic M-estimators we have presented above, 
statistical guarantees are, to the best of our knowledge, elusive.  
\end{remark}

\section{Algorithm}
\label{sec:4}

Recall from \eqref{e:prob0} that the problem of perspective
M-estimation is to
\begin{equation}
\label{e:prob1}
\minimize{s\in\RR^N,\,t\in\RR^P,\,b\in\RR^p}
\varsigma(s)+\varpi(t)+\theta(b)
+{\Sum_{i=1}^{N}\widetilde{\varphi}_i
\big(\sigma_i,X_ib-y_i\big)+\Sum_{i=1}^{P}
\widetilde{\psi}_i\big(\tau_i,L_ib\big)}.
\end{equation}
This minimization problem is quite complex, as it involves the
sum of several terms, compositions with linear operators, 
as well as perspective functions. In addition, 
none of the functions present in the model is assumed to have 
any full domain or smoothness property. 
In this section, we show that via
suitable reformulations in higher dimensional product spaces,
\eqref{e:prob1} can be reduced to a problem which is amenable to
Douglas-Rachford splitting and which, once reformulated in the
original space, produces a method which requires only to use
separately the proximity operators of the functions $\theta$, 
$(\widetilde{\varphi}_i)_{1\leq i\leq N}$, and
$(\widetilde{\psi}_i)_{1\leq i\leq P}$, the proximity operators
of the functions $\varsigma$ and $\varpi$, 
as well as application of
simple linear transformations. This method will be shown to
produce sequence $(s_k)_{k\in\NN}$, $(t_k)_{k\in\NN}$, and 
$(b_k)_{k\in\NN}$ which converge respectively to vectors $s$,
$t$, and $b$ that solve \eqref{e:prob1}.

Let us set $\varrho\colon\RR^N\times\RR^P\to\RX\colon(s,t)
\mapsto\varsigma(s)+\varpi(t)$, $M=N+P$, and 
\begin{equation}
(\forall i\in\{1,\ldots,N\})\quad
\begin{cases}
\vartheta_i=\varphi_i\\
m_i=n_i\\
w_i=y_i\\
A_i=X_i
\end{cases}
\quad\text{and}\quad
(\forall i\in\{N+1,\ldots,M\})\quad
\begin{cases}
\vartheta_i=\psi_{i-N}\\
m_i=p_{i-N}\\
w_i=0\\
A_i=L_{i-N}.
\end{cases}
\end{equation}
Then, upon introducing the variable
$v=(s,t)=(\nu_i)_{1\leq i\leq M}\in\RR^M$, we can rewrite 
\eqref{e:prob1} as
\begin{equation}
\label{e:prob72}
\minimize{v\in\RR^{M},\,b\in\RR^p}
{\varrho(v)+\theta(b)+\Sum_{i=1}^{M}\widetilde{\vartheta}_i
\big(\nu_i,A_ib-w_i\big)}.
\end{equation}
Now let us set $m=n+p$ and define 
\begin{equation}
\label{e:nancago12}
A=
\begin{bmatrix}
A_1\\
\vdots\\
A_M\\
\end{bmatrix}
\end{equation}
and 
\begin{equation}
\label{e:nancago4}
\begin{cases}
\begin{array}{ccll}
\boldsymbol{f}\colon&\RR^M\times\RR^p&\to&\RX\\[2mm]
&(v,b)\;&\mapsto&\varrho(v)+\theta(b)
\end{array}\\[5mm]
\begin{array}{ccll}
\boldsymbol{g}\colon&\RR^M\times\RR^{m}&\to&\RX\\
&\big(v,z)&\mapsto&
\Sum_{i=1}^M\widetilde{\vartheta}_i(\nu_i,z_i-w_i)
\end{array}\\[7mm]
\begin{array}{ccll}
\boldsymbol{L}\colon&\RR^M\times\RR^p&\to&
\RR^M\times\RR^m\\[2mm]
&(v,b)\;&\mapsto&\big(v,Ab\big).
\end{array}
\end{cases}
\end{equation}
Then, upon introducing the variable 
$\boldsymbol{a}=(v,b)\in\RR^M\times\RR^p$, \eqref{e:prob72} can
be rewritten as 
\begin{equation}
\label{e:prob73}
\minimize{\boldsymbol{a}\in\RR^{M+p}}
{\boldsymbol{f}(\boldsymbol{a})+
\boldsymbol{g}(\boldsymbol{L}\boldsymbol{a})},
\end{equation}
which we can solve by various algorithms \cite{Siop11,MaPr18}. 
Following an approach used in \cite{Jmaa18} and \cite{Siop15}, we
reformulate \eqref{e:prob73} as a problem involving the
sum of two functions $\boldsymbol{F}$ and $\boldsymbol{G}$, 
and then solve it via the Douglas-Rachford algorithm 
\cite{Livre1,Ecks92,Lion79}. To this end,
define 
\begin{equation}
\label{e:ramm3}
\boldsymbol{F}\colon\RR^{M+p}\times\RR^{M+m}\to\RX\colon
(\boldsymbol{a},\boldsymbol{c})\mapsto 
\boldsymbol{f}(\boldsymbol{a})+\boldsymbol{g}(\boldsymbol{c})
\end{equation}
and 
\begin{equation}
\label{e:ramm4}
\boldsymbol{G}=\iota_{\boldsymbol{V}}, 
\quad\text{where}\quad
\boldsymbol{V}=\Menge{(\boldsymbol{x},\boldsymbol{h})
\in\RR^{M+p}\times\RR^{M+m}}{\boldsymbol{Lx}=\boldsymbol{h}}
\end{equation}
is the graph of $\boldsymbol{L}$. Then, in terms of the variable
$\boldsymbol{u}=(\boldsymbol{a},\boldsymbol{c})$,
\eqref{e:prob73} is equivalent to 
\begin{equation}
\label{e:prob83}
\minimize{\boldsymbol{u}\in\RR^{2M+m+p}}
{\boldsymbol{F}(\boldsymbol{u})+\boldsymbol{G}(\boldsymbol{u})}.
\end{equation}
Let $\gamma\in\RPP$, let $\boldsymbol{v}_0\in\RR^{2M+p+m}$, and let
$(\mu_k)_{k\in\NN}$ be a sequence in $\left]0,2\right[$ such that 
$\sum_{k\in\NN}\mu_k(2-\mu_k)=\pinf$. The Douglas-Rachford 
algorithm for solving \eqref{e:prob83} is 
\cite[Section~28.3]{Livre1}
\begin{equation}
\label{e:kj18}
\begin{array}{l}
\text{for}\;k=0,1,\ldots\\
\left\lfloor
\begin{array}{l}
\boldsymbol{u}_{k}=\prox_{\gamma{\boldsymbol{G}}}\boldsymbol{v}_k\\
\boldsymbol{w}_{k}=\prox_{\gamma{\boldsymbol{F}}}
(2\boldsymbol{u}_{k}-\boldsymbol{v}_k)\\
\boldsymbol{v}_{k+1}=\boldsymbol{v}_k+
\mu_k(\boldsymbol{w}_{k}-\boldsymbol{u}_{k}).
\end{array}
\right.\\[2mm]
\end{array}
\end{equation}
Under the qualification condition 
\begin{equation}
\label{e:cq}
\boldsymbol{V}\cap\reli\dom\boldsymbol{F}\neq\emp,
\end{equation}
the sequence $(\boldsymbol{u}_k)_{k\in\NN}$ is guaranteed to 
converge to a solution $\boldsymbol{u}$ to \eqref{e:prob83} 
\cite[Corollary~27.4]{Livre1}. To make this algorithm more 
explicit, we first use \eqref{e:ramm3} and 
\cite[Proposition~24.11]{Livre1} to obtain 
\begin{equation}
\label{e:9hfe6-03e}
\prox_{\boldsymbol{F}}\colon (\boldsymbol{a},\boldsymbol{c})
\mapsto\big(\prox_{\boldsymbol{f}}\boldsymbol{a},
\prox_{\boldsymbol{g}}\boldsymbol{c}\big).
\end{equation}
Next, we derive from \eqref{e:ramm4} that $\prox_{\boldsymbol{G}}$
is the projection operator onto $\boldsymbol{V}$, that is
\cite[Example~29.19(i)]{Livre1},
\begin{equation}
\label{e:9hfe6-03f}
\prox_{\boldsymbol{G}}\colon
(\boldsymbol{x},\boldsymbol{h})\mapsto 
(\boldsymbol{a},\boldsymbol{La}),
\quad\text{where}\;\;
\boldsymbol{a}=\boldsymbol{x}-\boldsymbol{L}^\top
\big(\ID+\boldsymbol{LL}^\top\big)^{-1}
(\boldsymbol{Lx}-\boldsymbol{h}).
\end{equation}
Therefore, using the notation
\begin{equation}
\label{e:nancago8}
\boldsymbol{R}=\boldsymbol{L}^\top(\ID+\boldsymbol{LL}^\top)^{-1}
\quad\text{and}\quad
(\forall k\in\NN)\quad 
\begin{cases}
\boldsymbol{u}_k=(\boldsymbol{a}_k,\boldsymbol{c}_k)\\
\boldsymbol{v}_k=(\boldsymbol{x}_k,\boldsymbol{h}_k)\\
\boldsymbol{w}_k=(\boldsymbol{z}_k,\boldsymbol{d}_k),
\end{cases}
\end{equation}
we see that, given some initial points  
$\boldsymbol{x}_0\in\RR^{M+p}$ and $\boldsymbol{h}_0\in\RR^{m+M}$, 
\eqref{e:kj18} amounts to iterating
\begin{equation}
\label{e:9hfe6-03g}
\begin{array}{l}
\text{for}\;k=0,1,\ldots\\
\left\lfloor
\begin{array}{l}
\boldsymbol{q}_k=\boldsymbol{Lx}_k-\boldsymbol{h}_k\\
\boldsymbol{a}_{k}=\boldsymbol{x}_k-\boldsymbol{Rq}_k\\
\boldsymbol{c}_{k}=\boldsymbol{L}\boldsymbol{a}_{k}\\
\boldsymbol{z}_{k}=\prox_{\gamma\boldsymbol{f}}
(2\boldsymbol{a}_k-\boldsymbol{x}_k)\\
\boldsymbol{d}_{k}=\prox_{\gamma\boldsymbol{g}}
(2\boldsymbol{c}_k-\boldsymbol{h}_k)\\
\boldsymbol{x}_{k+1}=\boldsymbol{x}_k+\mu_k(\boldsymbol{z}_k-
\boldsymbol{a}_k)\\
\boldsymbol{h}_{k+1}=\boldsymbol{h}_k+
\mu_k(\boldsymbol{d}_k-\boldsymbol{c}_k).
\end{array}
\right.\\[2mm]
\end{array}
\end{equation}
In addition, it generates a sequence $(\boldsymbol{a}_k)_{k\in\NN}$
that converges to a solution $\boldsymbol{a}$ to \eqref{e:prob73}. 
Now set
\begin{equation}
\label{e:nancago11}
(\forall k\in\NN)\quad
\begin{cases}
\boldsymbol{a}_k=(s_k,t_k,b_k)\in\RR^N\times\RR^P\times\RR^p\\
\boldsymbol{c}_k=(s_k,t_k,c_{b,k})\in
\RR^N\times\RR^P\times\RR^{n+p}\\ 
\boldsymbol{x}_k=(x_{s,k},x_{t,k},x_{b,k})\in
\RR^N\times\RR^P\times\RR^p\\ 
\boldsymbol{h}_k=(h_{s,k},h_{t,k},h_{b,k})\in
\RR^N\times\RR^P\times\RR^{n+p}\\ 
\boldsymbol{z}_k=(z_{s,k},z_{t,k},z_{b,k})\in
\RR^N\times\RR^P\times\RR^p\\ 
\boldsymbol{d}_k=(d_{s,k},d_{t,k},d_{b,k})\in
\RR^N\times\RR^P\times\RR^{n+p}\\ 
\boldsymbol{q}_k=(q_{s,k},q_{t,k},q_{b,k})\in
\RR^N\times\RR^P\times\RR^{n+p},
\end{cases}
\end{equation}
and observe that \eqref{e:nancago4} and \eqref{e:nancago8} yield
\begin{equation}
(\forall k\in\NN)\quad
\boldsymbol{R}\boldsymbol{q}_k=\Big(q_{s,k}/2,q_{t,k}/2,
Qq_{b,k}\Big),
\quad\text{where}\quad
Q=A^\top(\Id+{AA}^\top)^{-1}.
\end{equation}
Let us further decompose the above vectors as 
\begin{equation}
\begin{cases}
s_{k}=(\sigma_{1,k},\ldots,\sigma_{N,k})\in\RR^N\\
t_{k}=(\tau_{1,k},\ldots,\tau_{P,k})\in\RR^P\\
h_{s,k}=(\eta_{1,k},\ldots,\eta_{N,k})\in\RR^N\\
h_{t,k}=(\eta_{N+1,k},\ldots,\eta_{N+P,k})\in\RR^P\\
h_{b,k}=(h_{1,k},\ldots,h_{n,k},h_{n+1,k},\ldots,h_{n+p,k})\in
\RR^{n_1}\times\cdots\times\RR^{n_N}\times\RR^{p_1}\times
\cdots\times\RR^{p_P}\\
c_{b,k}=(c_{1,k},\ldots,c_{n,k},c_{n+1,k},\ldots,c_{n+p,k})\in
\RR^{n_1}\times\cdots\times\RR^{n_N}\times\RR^{p_1}\times
\cdots\times\RR^{p_P}\\
q_{b,k}=(q_{1,k},\ldots,q_{n,k},q_{n+1,k},\ldots,q_{n+p,k})\in
\RR^{n_1}\times\cdots\times\RR^{n_N}\times\RR^{p_1}\times
\cdots\times\RR^{p_P}\\
d_{s,k}=(\delta_{1,k},\ldots,\delta_{N,k})\in\RR^N\\
d_{t,k}=(\delta_{N+1,k},\ldots,\delta_{N+P,k})\in\RR^P\\
d_{b,k}=(d_{1,k},\ldots,d_{N,k},d_{N+1,k},\ldots,d_{N+P,k})\in
\RR^{n_1}\times\cdots\times\RR^{n_N}\times\RR^{p_1}\times
\cdots\times\RR^{p_P}.
\end{cases}
\end{equation}
Then, given $x_{s,0}\in\RR^{N}$,
$x_{t,0}\in\RR^{P}$,
$x_{b,0}\in\RR^{p}$,
$h_{s,0}\in\RR^{N}$,
$h_{t,0}\in\RR^{P}$, and
$h_{b,0}\in\RR^{m}$,
\eqref{e:9hfe6-03g} consists in iterating
\begin{equation}
\label{e:9hyR09}
\begin{array}{l}
\text{for}\;k=0,1,\ldots\\
\left\lfloor
\begin{array}{l}
q_{s,k}=x_{s,k}-h_{s,k}\\
q_{t,k}=x_{t,k}-h_{t,k}\\
q_{b,k}=Ax_{b,k}-{h}_{b,k}\\
\text{for}\;i=1,\ldots,N\\
\left\lfloor
\begin{array}{l}
q_{i,k}=X_ix_{b,k}-{h}_{i,k}
\end{array}
\right.\\[2mm]
\text{for}\;i=1,\ldots,P\\
\left\lfloor
\begin{array}{l}
q_{N+i,k}=L_ix_{b,k}-{h}_{N+i,k}
\end{array}
\right.\\[2mm]
s_{k}=x_{s,k}-q_{s,k}/2\\
t_{k}=x_{t,k}-q_{t,k}/2\\
b_{k}=x_{b,k}-Qq_{b,k}\\
z_{s,k}=\prox_{\gamma\varsigma}(2s_k-x_{s,k})\\
z_{t,k}=\prox_{\gamma\varpi}(2t_k-x_{t,k})\\
z_{b,k}=\prox_{\gamma\theta}(2b_k-x_{b,k})\\
x_{s,k+1}=x_{s,k}+\mu_k({z}_{s,k}-s_k)\\
x_{t,k+1}=x_{t,k}+\mu_k({z}_{t,k}-t_k)\\
x_{b,k+1}=x_{b,k}+\mu_k(z_{b,k}-b_k)\\
\text{for}\;i=1,\ldots,N\\
\left\lfloor
\begin{array}{l}
c_{i,k}=X_ib_k\\
(\delta_{i,k},d_{i,k})
=(0,y_i)+\prox_{\gamma\widetilde{\varphi}_i}
(2\sigma_{i,k}-\eta_{i,k},2c_{i,k}-h_{i,k}-y_i)\\
\end{array}
\right.\\[2mm]
\text{for}\;i=1,\ldots,P\\
\left\lfloor
\begin{array}{l}
c_{N+i,k}=L_ib_k\\
(\delta_{N+i,k},d_{N+i,k})=\prox_{\gamma\widetilde{\psi}_i}
(2\tau_{i,k}-\eta_{N+i,k},2c_{N+i,k}-h_{N+i,k})\\
\end{array}
\right.\\[2mm]
h_{s,k+1}=h_{s,k}+\mu_k(d_{s,k}-s_k)\\
h_{t,k+1}=h_{t,k}+\mu_k(d_{t,k}-t_k)\\
h_{b,k+1}=h_{b,k}+\mu_k(d_{b,k}-c_{b,k}).
\end{array}
\right.\\[2mm]
\end{array}
\end{equation}
Using the above mentioned results for the convergence of the
sequence $(\boldsymbol{u}_k)_{k\in\NN}$ produced by
\eqref{e:kj18}, we obtain in the setting
of Problem~\ref{prob:1} the convergence of the 
sequences $(s_k)_{k\in\NN}$, $(t_k)_{k\in\NN}$, and 
$(b_k)_{k\in\NN}$ generated by \eqref{e:9hyR09} to vectors $s$,
$t$, and $b$, respectively, that solve \eqref{e:prob1}.

\section{Numerical experiments}
\label{sec:5}
We illustrate the versatility of perspective M-estimation for
sparse robust regression in a number of numerical experiments.  The
algorithm outlined in Section \ref{sec:4} has been implemented for
several important instances in MATLAB and is available at
\url{https://github.com/muellsen/PCM}. We set
$\mu=1.9$ and $\gamma=1$ for all model instances. We declare that
the algorithm has converged at iteration $k$ if
$\|b_k-b_{k+1}\|_2<\epsilon$, for some $\epsilon\in\RPP$ to be
specified.

\begin{figure}[h]
\centering
\includegraphics[width=14cm]{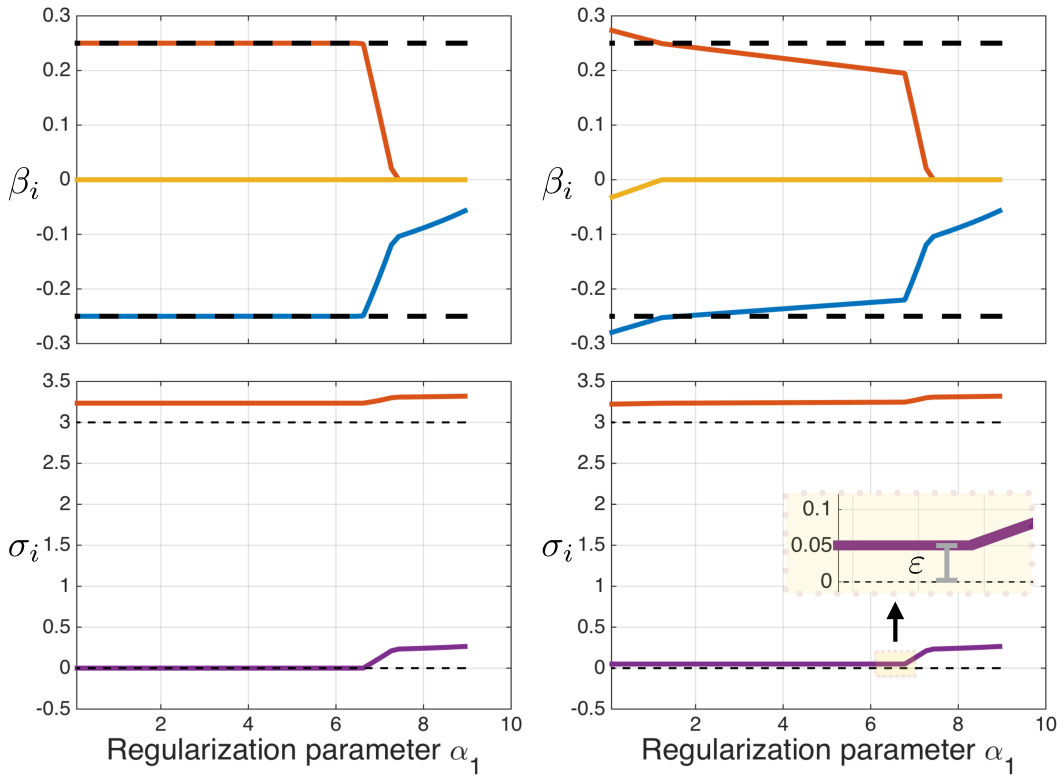}
\caption{Generalized heteroscedastic lasso solutions for  
$b$ (top panel) and $s$ (bottom panel) across the
$\alpha_1$-path. The left panels show the results of fully
non-smooth perspective M-estimation. The right panels show the
results for a smoothed version
with $\sigma_i\in\left[\varepsilon,\pinf\right[$ for
$\varepsilon=0.05$. The dashed lines mark the ground truth entries
of $\overline{b}$ and $\overline{s}$. \label{fig:nonsmooth}}
\end{figure}

\subsection{Numerical illustrations on low-dimensional data}
Our algorithmic approach to perspective M-estimation can
effortlessly handle non-smooth data fitting terms. To illustrate
this property, we consider a partially noiseless data formation
model in low dimensions. We instantiate the data model \eqref{e:1} as
follows. We consider the design matrix $X\in\RR^{n\times p}$ with
$p=3$ and sample size $n=18$. Entries in the design matrix and the
noise vector $e\in\RR^n$ are sampled from a standard normal
distribution $\mathcal{N}(0,1)$.  The matrix $C\in\RR^{n\times n}$
is a diagonal matrix with $N=2$ groups.  We set 
$\overline{s}=[\overline{\sigma}_1, \overline{\sigma}_2]^\top=
[3,0]^\top$.  The $i$th diagonal entry of $C$
is set to $\overline{\sigma}_1$ for $i\in\{1,\ldots,8\}$
and to $\overline{\sigma}_2$ for $i\in\{10,\ldots,18\}$, resulting
in noise-free observations for the second group. The mean shift
(or outlier) vector is $\overline{o}=[0,\ldots,0]^\top$. The
regression
vector is $\overline{b}=[0.25,-0.25,0]^\top$. The goal is to
estimate the regression vector $\overline{b}\in\RR^3$ as well as the
concomitant (or scale) vector 
$\overline{s}=[\overline{\sigma}_1,\overline{\sigma}_2]^\top$.  We
consider the generalized heteroscedastic lasso of
Example~\ref{ex:24a} with $N=2$ and $q=2$. To demonstrate
the advantage of our non-smooth approach in this partially
noiseless setting, we consider two variations of the model, the
standard non-smooth case and a smoothed version with 
$D=\left[\varepsilon,\pinf\right[^2$
for $\varepsilon=0.05$. The convergence accuracy 
for the algorithm is set to $\epsilon=10^{-8}$.
Figure~\ref{fig:nonsmooth} shows 
the estimates $b=[\beta_1,\beta_2,\beta_3]^\top$ 
and $s=[\sigma_1,\sigma_2]^\top$ across the regularization path, 
where $\alpha_1\in\{0.089,\ldots,8.95\}$, with $200$ values
equally spaced on a log-linear grid for both settings. 
The results indicate that only the heteroscedastic lasso in the
non-smooth setting can recover the ground truth regression vector
$\overline{b}$ (top left panel) and $\overline{\sigma}_2=0$
(bottom left
panel). In both settings the estimate $\sigma_1$ is slightly
overestimated (due to the finite sample size).  The ``smoothed''
version of the heteroscedastic lasso cannot achieve exact recovery
of $\overline{b}$ across the regularization path (top right panel).

\subsection{Numerical illustrations for correlated designs and
outliers}
To illustrate the efficacy of the different M-estimators we 
instantiate the full data formation model \eqref{e:1} as
follows. We consider the design matrix $X\in\RR^{n\times p}$ with
$p=64$ and sample size $n=75$ where each row $X_i$ is sampled
from a correlated normal distribution $\mathcal{N}(0,\Sigma)$
with off-diagonal entries $0.3$ and diagonal entries
$1$. The entries of $e\in\RR^n$ are realizations of
i.i.d. zero mean normal variables $\mathcal{N}(0,1)$. The matrix
$C\in\RR^{n\times n}$ is a diagonal matrix with three groups. 
We set $\overline{s}=[\overline{\sigma}_1, \overline{\sigma}_2,
\overline{\sigma}_3]^\top =[5,0.5,0.05]^\top$.
The $i$th diagonal element of $C$ is set to 
$\overline{\sigma}_1$ for $i\in\{1,\ldots,25\}$, to 
$\overline{\sigma}_2$ for $i\in\{26,\ldots,50\}$, and to 
$\overline{\sigma}_3$ for $i\in\{51,\ldots,75\}$. The mean shift
vector $o\in\RR^n$ contains $\lceil 0.1 n \rceil=8$
non-zero entries, sampled from $\mathcal{N}(0,5)$. The entries of
the regression vector $\overline{b}\in\RR^p$ are set to
$\overline{\beta}_i=-1$ for $i\in\{1,3,5\}$ and
$\overline{\beta}_i= 1$ for $i\in\{2,4,6\}$.  

The presence of outliers, correlation in the design, and
heteroscedasticity provides a considerable challenge for regression
estimation and support recovery with standard models such as the
lasso. We consider instances of the perspective
M-estimation model of increasing complexity that can cope with
various aspects of the data formation model. Specifically, we 
use the models outlined in Examples~\ref{ex:24a} and
\ref{ex:24b} (with $\alpha_2=0$) in homoscedastic and
heteroscedastic mode. For all models, we compute the minimally
achievable mean absolute error (MAE) $\|X b- X\overline{b}\|_1/n$
across the $\alpha_1$-regularization path, where $\alpha_1 
\in\{0.254,\ldots,25.42\}$, with $50$ values
equally spaced on a log-linear grid. The convergence criterion is
$\epsilon=5\cdot10^{-4}$.

\textbf{Homoscedastic models.}
We first consider homoscedastic instances of Examples~\ref{ex:24a}
and \ref{ex:24b}, in which we jointly estimate a regression vector 
and a single concomitant parameter in the data fitting part.
We consider the generalized scaled lasso of
Example~\ref{ex:24a} and the generalized Huber of
Example~\ref{ex:24b} with exponents $q\in\{3/2,2\}$.
Figure~\ref{fig:predhomo} presents 
the estimation results of $b$ over the relevant $\alpha_1$-path.

\begin{figure}[t]
\centering
\includegraphics[width=14cm]{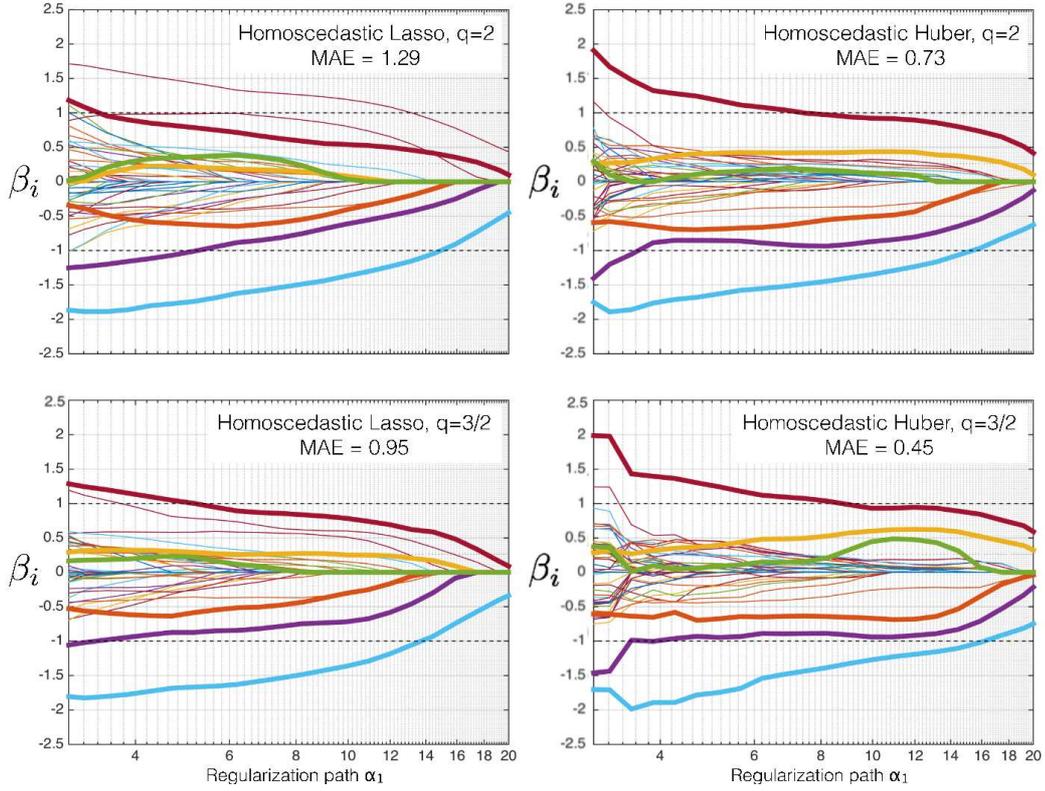}
\caption{Generalized homoscedastic lasso and Huber solutions for
$b$ when $q=2$ (top panel) and $q=3/2$ (bottom panel)
across the relevant $\alpha_1$-path. The minimally achievable mean
absolute error (MAE) is shown for all models. The six highlighted
$\beta_i$ trajectories mark the true non-zeros entries of
$\overline{b}$. The dashed black lines show the values of the true
$\overline{b}$ entries.
\label{fig:predhomo}}
\end{figure}

\textbf{Heteroscedastic models.}
We consider the same model instances as previously described
but in the heteroscedastic setting. We jointly estimate regression
vectors and concomitant scale parameters for each of the three
groups.  
Figure~\ref{fig:predhet} presents the results for heteroscedastic
lasso and Huber estimations of $b$ across the relevant
$\alpha_1$-path. The convergence criterion is
$\epsilon=10^{-4}$.

\begin{figure}[t]
\centering
\includegraphics[width=14cm]{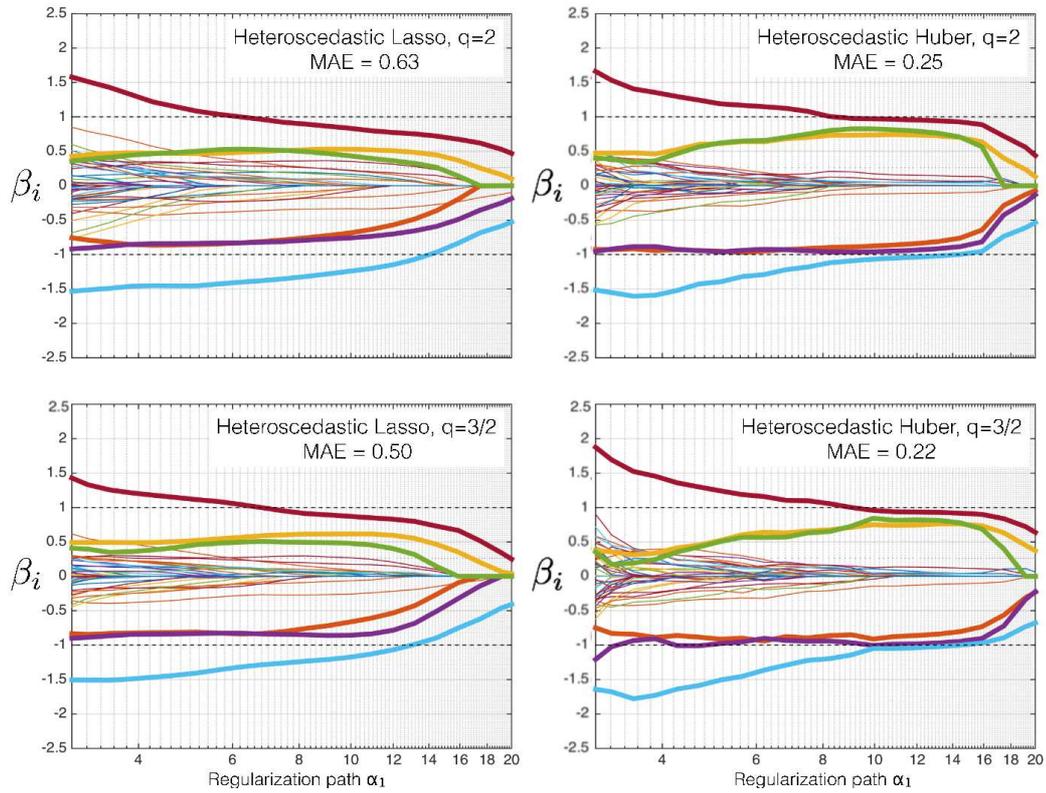}
\caption{Generalized heteroscedastic lasso and Huber solutions for
$\overline{b}$ when $q=2$ (top panel) and $q=3/2$ (bottom panel)
across the relevant
$\alpha_1$-path. The minimally achievable mean absolute error
(MAE) is shown for all models. The six highlighted $\beta_i$ 
trajectories mark the true non-zeros entries of $\overline{b}$. The 
dashed black lines show the values of the true $\overline{b}$ 
entries.
\label{fig:predhet}}
\end{figure}

The numerical experiments indicate that only heteroscedastic
M-estimators are able to produce convincing $b$ estimates (as
captured by lower MAE). The heteroscedastic Huber model with
$q=3/2$ (see Figure~\ref{fig:predhomo} lower right panel) achieves
the best performance in terms of MAE among all tested models.

\begin{figure}[t]
\centering
\includegraphics[width=15cm]{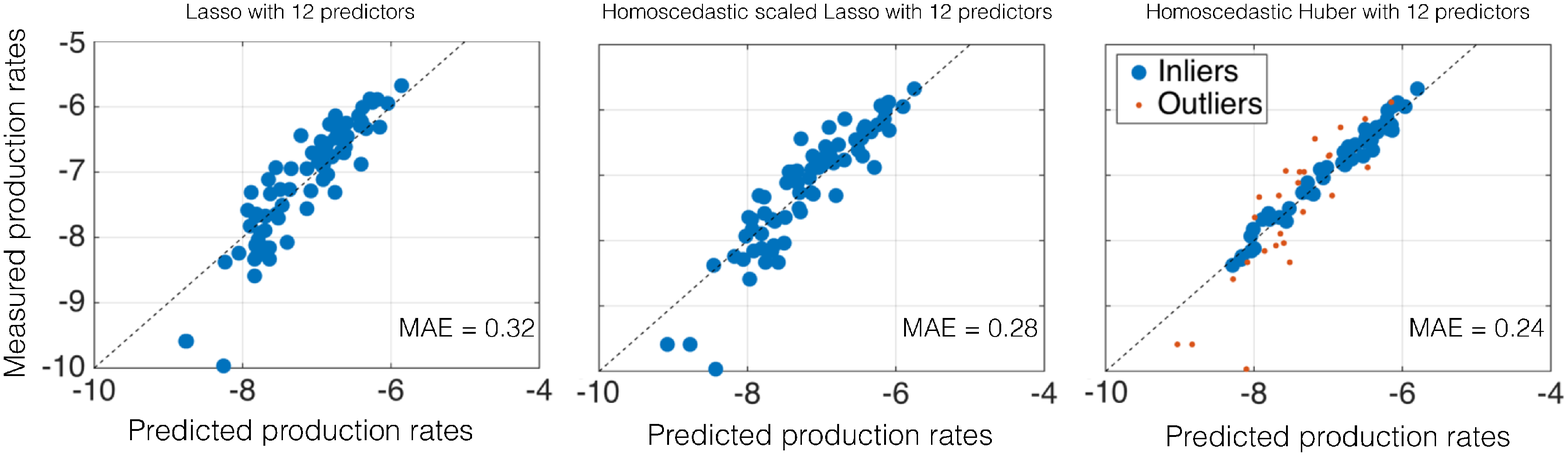}
\caption{Standard lasso, homoscedastic lasso, and homoscedastic
Huber log-production rate predictions with identical model
complexity (every vector $b$ comprises twelve non-zero components)
and associated MAE (computed across all samples). The Huber
M-estimator identifies 26 outliers (marked in red). 
\label{fig:bsubtilis}}
\end{figure}

\subsection{Robust regression for gene expression data}
We consider a high-dimensional linear regression problem
from genomics \cite{Buhlmann2014}. The design matrix $X$ consists
of $p=4088$ highly correlated gene expression profiles for $n =
71$ different strains of Bacillus subtilis (B. subtilis). The
response $y\in\RR^{71}$ comprises standardized riboflavin
(Vitamin B) log-production rates for each strain. The statistical
task is to identify a small set of genes that is highly
predictive of the riboflavin production rate. No grouping of the
different strain measurements is available. We thus consider the
homoscedastic models from Example~\ref{ex:4} with $\alpha_2=0$ and 
Example~\ref{ex:24b} with $\alpha_2=0$. We
optimize the corresponding perspective M-estimation models over
the $\alpha_1$-path where $\alpha_1=[0.623,\ldots,6.23]$ with 20
values
equally spaced on a log-linear grid. We compare the resulting
models with the standard lasso in terms of in-sample prediction
performance. Figure \ref{fig:bsubtilis} summarizes the results
for the in-sample prediction of the three different models with
identical model complexity (twelve non-zero entries in
$\overline{b}$).
To assess model quality, we compute the minimally achievable mean
absolute error (MAE) $\|Xb-y\|_1/n$ for these three
models. The Huber model achieves significantly improved MAE
(0.24) compared to lasso (0.32). The Huber models also identifies
26 non-zero components in the outlier vector $o$ 
(shown in red in the rightmost panel of Figure~\ref{fig:bsubtilis}).

\end{document}